\documentclass[11pt, twoside, leqno]{article}
\usepackage{amssymb}
\usepackage{amsmath}
\usepackage{amsthm}
\usepackage{color}
\usepackage{mathrsfs}
\usepackage{graphicx}
\usepackage{indentfirst}
\usepackage{txfonts}

\allowdisplaybreaks

\def\XXint#1#2#3{{\setbox0=\hbox{$#1{#2#3}{\int}$ }
\vcenter{\hbox{$#2#3$ }}\kern-.6\wd0}}

\newcommand{\R}{\mathbb{R}}
\newcommand{\Z}{\mathbb{Z}}
\newcommand{\N}{\mathbb{N}}

\newtheorem{theorem}{Theorem}[section]
\newtheorem{lemma}[theorem]{Lemma}

\newtheorem{proposition}[theorem]{Proposition}

\theoremstyle{definition}
\newtheorem{remark}[theorem]{Remark}

\renewcommand{\appendix}{\par
   \setcounter{section}{0}%
   \setcounter{subsection}{0}%
   \setcounter{subsubsection}{0}%
   \gdef\thesection{\@Alph\c@section}%
   \gdef\thesubsection{\@Alph\c@section.\@arabic\c@subsection}%
   \gdef\theHsection{\@Alph\c@section.}%
   \gdef\theHsubsection{\@Alph\c@section.\@arabic\c@subsection}%
   \csname appendixmore\endcsname
 }

\numberwithin{equation}{section}

\begin{document}

\arraycolsep=1pt

\title{\bf\Large Mixed norm estimates for dilated averages over planar curves
\footnotetext{\hspace{-0.35cm} 2020 {\it
Mathematics Subject Classification}. Primary 42B10;
Secondary 42B15.
\endgraf {\it Key words and phrases.} generalised Radon transforms, half-wave propagator, Fourier integral operators, mixed norm estimates, local smoothing estimates.
}}
\author{Junfeng Li, Zengjian Lou and Haixia Yu\footnote{Corresponding author.}}

\date{}

\maketitle

\vspace{-0.7cm}

\begin{abstract}
 In this paper, we investigate the mixed norm estimates for the operator $ T $associated with a dilated plane curve $(ut, u\gamma(t))$, defined by
\[
Tf(x, u) := \int_{0}^{1} f(x_1 - ut, x_2 - u\gamma(t)) \, dt,
\]
where $ x := (x_1, x_2) $ and $\gamma $ is a general plane curve satisfying appropriate smoothness and curvature conditions. More precisely, we establish the $ L_x^p(\mathbb{R}^2) \rightarrow L_x^q L_u^r(\mathbb{R}^2 \times [1, 2]) $ (space-time) estimates for $ T $, whenever $(\frac{1}{p},\frac{1}{q})$ satisfy
\[
\max\left\{0, \frac{1}{2p} - \frac{1}{2r}, \frac{3}{p} - \frac{r+2}{r}\right\} < \frac{1}{q} \leq \frac{1}{p} < \frac{r+1}{2r}
\]
and
$$1 + (1 + \omega)\left(\frac{1}{q} - \frac{1}{p}\right) > 0,$$
where $ r \in [1, \infty] $ and $ \omega := \limsup_{t \rightarrow 0^+} \frac{\ln|\gamma(t)|}{\ln t} $. These results are sharp, except for certain borderline cases. Additionally, we examine the $ L_x^p(\mathbb{R}^2) \rightarrow L_u^r L_x^q(\mathbb{R}^2 \times [1, 2]) $ (time-space) estimates for $T $, which are especially almost sharp when $p=2$ or $p\in [1, \frac{3}{2}]\cup [4, \infty]$.
\end{abstract}

\section{Introduction}\label{section:1}

Let $\gamma \in C^{N}(0,1]$ with $ N \in \mathbb{N}$ sufficiently large. We define the operator $T$ along the dilated plane curve $(ut, u\gamma(t))$ as
\[
Tf(x, u) := \int_{0}^{1} f(x_1 - ut, x_2 - u\gamma(t)) \, dt,
\]
which has been extensively studied in works such as Hickman \cite{H}, Beltran, Guo, Hickman, and Seeger \cite{BGHS}, Ko, Lee, and Oh \cite{KLO}, among others. We also consider the maximal operator
\[
M_\gamma f(x) := \sup_{u \in (0, \infty)} \left| Tf(x, u) \right|.
\]
Bourgain \cite{Bour86} established foundational estimates for maximal averages over planar convex curves. Beltran, Guo, Hickman, and Seeger \cite{BGHS}, as well as Ko, Lee, and Oh \cite{KLO}, demonstrated that the maximal average over a curve in three dimensions is bounded on $L_x^p(\mathbb{R}^3)$ if and only if $p > 3$. For higher-dimensional settings, we refer to Ko, Lee and Oh \cite{KLO23}. Related studies on maximal functions associated with hypersurfaces can be found in \cite{SSte, SoSte, IKM} and references therein.

A natural question arises concerning the $ L_x^p(\mathbb{R}^2) $ to $ L_x^q(\mathbb{R}^2) $ boundedness of $ M_\gamma $. It is well-established in various contexts that $ M_\gamma $ is not $ L^p \rightarrow L^q $ bounded unless $ p = q $ (see Liu and Yu \cite[Remark 1.3]{LiuYu}). However, when the supremum is restricted to $ u \in [1, 2] $, they derived an almost sharp $ L_x^p(\mathbb{R}^2) \rightarrow L_x^q(\mathbb{R}^2) $ estimate for $ M_\gamma $, i.e., the $L_x^p(\mathbb{R}^2)\rightarrow L_x^qL^{\infty}_u(\mathbb{R}^2\times [1,2])$ estimate for $T$. This phenomenon, known as $ L^p $-improving, is discussed in greater detail in \cite{TW, Sch, BDH24, LWZ}.

In this paper, we aim to extend the estimates for the maximal function to mixed norm estimates for $ T $. Our primary objective is to establish $ L^p_x(\mathbb{R}^2) \rightarrow L^q_x L^r_u(\mathbb{R}^2 \times [1, 2]) $ (space-time) estimates for $ T $, which is related to $r$-variation estimates for $ T $ (see \cite{BORSS, LaY}). When $ q = r $, there is already a substantial body of literature addressing this case. For example, when $ \gamma(t) := t^2 $, the $ L_x^2(\mathbb{R}^2) \rightarrow L_{x,u}^6(\mathbb{R}^2 \times [1, 2]) $ boundedness of $ T $ was essentially established by Strichartz \cite{Str} and later refined by Schlag and Sogge \cite{SchS}. Additional results on $ L_x^p(\mathbb{R}^2) \rightarrow L_{x,u}^q(\mathbb{R}^2 \times [1, 2]) $ boundedness for $ T $ can be found in Gressman \cite{Gr1, Gr2}. More recently, Li, Liu, Lou, and Yu \cite{LLLY} established almost sharp $ L_x^p(\mathbb{R}^2) \rightarrow L_{x,u}^q(\mathbb{R}^2 \times [1, 2]) $ estimates for $ T $ along a broader class of curves. In higher dimensions, Hickman \cite{H} obtained nearly sharp $ L_x^p(\mathbb{R}^n) \rightarrow L_{x,u}^q(\mathbb{R}^n \times [1, 2]) $ estimates for $ T $ along the moment curve $ (t, t^2, \dots, t^n) $. He also raised the problem of determining the range of $ (\frac{1}{p}, \frac{1}{q}, \frac{1}{r}) $ for which mixed norm estimates for $ T $ hold.

The main purpose of this paper is to consider the mixed norm estimates for $ T $ in two-dimensional case, but
with some general plane curves including the parabola $\gamma(t):=t^2$. Our first result is the following.

\begin{theorem}\label{thm1}
Assume $\gamma \in C^{N}(0,1]$ with $N \in \mathbb{N}$ sufficiently large, $\lim_{t \rightarrow 0^+} \gamma(t) = 0$, and $\gamma$ is monotonic on $(0,1]$. Additionally, suppose $\gamma$ satisfies the following conditions:
\begin{enumerate}\label{curve gamma}
  \item[\rm(i)] there exist positive constants $\{C^{(j)}_{1}\}_{j=1}^{2}$ such that $|\frac{t^{j}\gamma^{(j)}(t)}{\gamma(t)}|\geq C^{(j)}_{1}$  for any $t\in (0,1]$;
  \item[\rm(ii)] there exist positive constants $\{C^{(j)}_{2}\}_{j=1}^{N}$ such that $|\frac{t^{j}\gamma^{(j)}(t)}{\gamma(t)}|\leq C^{(j)}_{2}$ for any $t\in (0,1]$.
\end{enumerate}
Then, there exists a positive constant $C$ depending only on $\gamma$, $p$, $q$ and $r$, such that for all $f \in L_x^{p}(\mathbb{R}^{2})$,
\[
\left\|Tf\right\|_{L_x^q L^r_u(\mathbb{R}^2 \times [1,2])} \leq C \|f\|_{L_x^{p}(\mathbb{R}^{2})},
\]
where $(\frac{1}{p}, \frac{1}{q})$ satisfies
\[
\max\left\{0, \frac{1}{2p} - \frac{1}{2r}, \frac{3}{p} - \frac{r+2}{r}\right\} < \frac{1}{q} \leq \frac{1}{p} < \frac{r+1}{2r}
\]
and
$$1 + (1 + \omega)\left(\frac{1}{q} - \frac{1}{p}\right) > 0$$
with $r \in [1, \infty]$. Moreover, the range of $(\frac{1}{p}, \frac{1}{q})$ is sharp except for some borderlines for all $r \in [1, \infty]$. Here and hereafter, $\omega := \limsup_{t \rightarrow 0^{+}} \frac{\ln|\gamma(t)|}{\ln t}$.
\end{theorem}

\begin{remark}
Here are some examples of curves satisfying the conditions (i) and (ii) of Theorem \ref{thm1}. We may add a characteristic function $\chi_{(0,\epsilon_0]}(t)$ if necessary, where $\epsilon_0$ is small enough.
\begin{enumerate}
\item[\rm(1)] $\gamma_1(t):=t^d$, where $d\in(0,\infty)$ and $d\neq1$;
\item[\rm(2)] $\gamma_2(t):=a_dt^d+a_{d+1}t^{d+1}+\cdots+a_{d+m}t^{d+m}$, where $d\geq 2$, $d\in\mathbb{N}$ and $m\in\mathbb{N}_0$ , $a_d\neq 0$, i.e., $\gamma_2$ is a polynomial of degree at least $d$ with no linear term and constant term;
\item[\rm(3)] $\gamma_3(t):=\sum_{i=1}^{d}\beta_i t^{\alpha_i}$, where $\alpha_i\in(0,\infty)$ for all $i=1,2, \cdots, d$, $\min_{i\in\{1,2, \cdots, d\}}\{\alpha_i\}_{i=1}^{d}\neq 1$ and $d\in \mathbb{N}$;
\item[\rm(4)] $\gamma_4(t):=1-\sqrt{1-t^2}$, $t\sin t$, $t-\sin t$, $1-\cos t$, or $e^t-t-1$, or $t^{d}\ln(1+t)$ with $d\in(0,\infty)$;
\item[\rm(5)] $\gamma_5(t)$ is a smooth function on $[0,1]$ satisfying $\gamma(0)=\gamma'(0)=\cdots=\gamma^{(d-1)}(0)=0$ and $\gamma^{(d)}(0)\neq0$, where $d\geq 2$ and $d\in\mathbb{N}$. From Iosevich \cite{Iose}, we know that $\gamma_5$ is finite type $d$ at $t=0$. $\gamma_2$ and $\gamma_4$ provide some special cases of $\gamma_5$.
\end{enumerate}
\end{remark}

\begin{remark}
The range of $(\frac{1}{p}, \frac{1}{q})$ in Theorem \ref{thm1} for all $r \in [1, \infty]$ is illustrated in Figure \ref{Figure:1}. This range corresponds to the open convex hull of the points $\{\textrm{O, A, B, C, D}\}$ but with the open line segment $(\textrm{O, A})$. From Figure \ref{Figure:1}, it is evident that the range of $(\frac{1}{p}, \frac{1}{q})$ diminishes as $r$ increases. Specifically, when $r = 1$, the range encompasses the entire lower triangle except for the boundary lines $\frac{1}{p} = 1$ and $\frac{1}{q} = 0$. Conversely, when $r = \infty$, the range coincides with the result obtained by Liu and Yu \cite{LiuYu}.
Additionally, the range of $(\frac{1}{p}, \frac{1}{q})$ is observed to shrink as $\omega$ increases. Notably, the condition $1 + (1 + \omega)(\frac{1}{q} - \frac{1}{p}) > 0$ does not impose any restrictions on the range of $(\frac{1}{p}, \frac{1}{q})$ when $\omega \leq \frac{4(r - 1)}{r + 4}$.
\begin{figure}[htbp]
\centering
\includegraphics[width=4in]{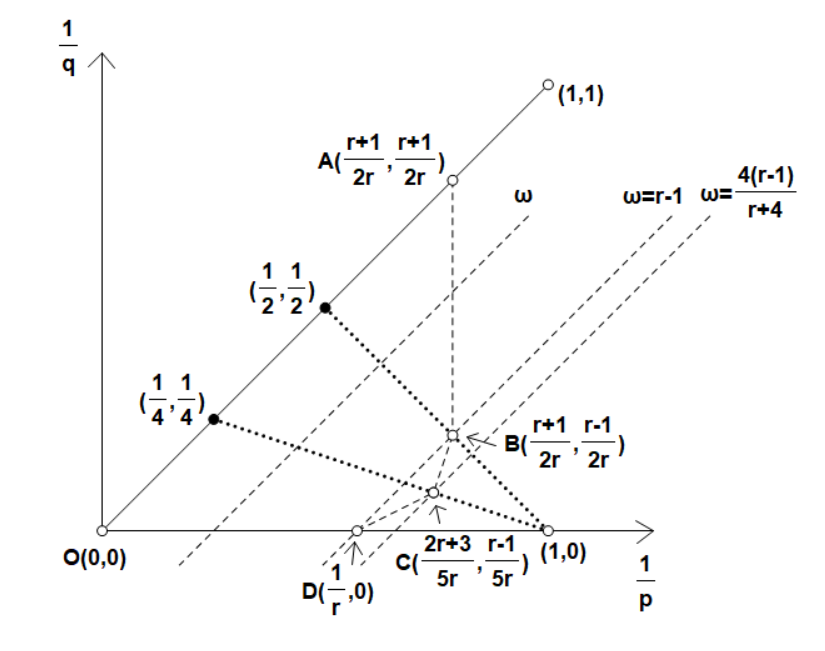}\
\caption{The range of $(\frac{1}{p}, \frac{1}{q})$ in Theorem \ref{thm1} for all $r \in [1, \infty]$.}\label{Figure:1}
\end{figure}
\end{remark}

Another motivation for this paper stems from the study of the maximal function $\mathcal{M}$, defined by averaging over the curve $(t, \gamma(t))$, as follows:
\begin{align*}
\mathcal{M}f(x) := \sup_{\epsilon \in (0, \infty)} \frac{1}{2\epsilon} \int_{-\epsilon}^{\epsilon} \left| f(x_1 - t, x_2 - \gamma(t)) \right| \, \mathrm{d}t.
\end{align*}
This operator, along with related families, has been extensively studied. For the case $\gamma(t) := t^2$, Nagel, Riviere, and Wainger \cite{NRW76} demonstrated that $\mathcal{M}$ is bounded on $L_x^p(\mathbb{R}^2)$ for all $p \in (1, \infty]$. Stein \cite{Ste1} extended this result to homogeneous curves, while Stein and Wainger \cite{SW1} investigated smooth curves. For more general families of curves, we refer readers to \cite{SW,CCVWW,CVWW}.

Let $\bar{\psi}: \mathbb{R}^+ \rightarrow \mathbb{R}$ be a smooth function supported on $\{ t \in \mathbb{R}: \frac{1}{8} \leq t \leq 8 \}$, satisfying $0 \leq \bar{\psi}(t) \leq 1$ and $\bar{\psi}(t) = 1$ on $\{ t \in \mathbb{R}: \frac{1}{4} \leq t \leq 4 \}$. Throughout this paper, we denote by $P_k$ the operator on $\mathbb{R}^n$ with multiplier $\bar{\psi}_k(|\xi|) := \bar{\psi}(2^{-k}|\xi|)$, where $k \in \mathbb{N}$. A key estimate required to establish Theorem \ref{thm1} is essentially the following almost sharp space-time estimate for the half-wave propagator:
\begin{align}\label{eq:1.2}
2^{-\frac{1}{2}k} \left\| e^{iu\sqrt{-\Delta}} P_k f \right\|_{L_x^q L^r_u(\mathbb{R}^2 \times [1,2])} \lesssim 2^{-\epsilon k} \| f \|_{L_x^p(\mathbb{R}^2)},
\end{align}
where $\epsilon$ is a positive constant. It is well known that the solution $g$ of the Cauchy problem for the wave equation in $\mathbb{R}^n\times \mathbb{R}$
\begin{equation*}
\begin{cases}
 (\partial^2_u-\Delta)g(x,u)=0, ~~(x,u)\in \mathbb{R}^n\times \mathbb{R}; \\
g(x,0)=f(x), ~~\partial_u g(x,0)=0,
\end{cases}
\end{equation*}
can be written in terms of the half-wave propagator $e^{iu\sqrt{-\Delta}} f$ as
\begin{equation*}
g(x,u)=\frac{1}{2}\left(e^{iu\sqrt{-\Delta}} f(x) +e^{-iu\sqrt{-\Delta}} f (x)\right).
\end{equation*}
Let $\dot{H}_x^s(\mathbb{R}^n)$ denote the homogeneous Sobolev space. The following space-time estimate for the half-wave propagator is of particular interest in harmonic analysis:
\begin{align}\label{eq:1.3}
\left\| e^{iu\sqrt{-\Delta}} f \right\|_{L_x^q L^r_u(\mathbb{R}^n \times \mathbb{R})} \lesssim \| f \|_{\dot{H}_x^s(\mathbb{R}^n)}.
\end{align}
In the present paper, we focus on this estimate for all $f \in L_x^p(\mathbb{R}^2)$ with $p \in [1, \infty]$.

The estimate \eqref{eq:1.3} and its local version have been studied by numerous authors. For instance, when $r = \infty$, \eqref{eq:1.3} represents a maximal function estimate for the half-wave propagator, which is associated with the pointwise convergence of solutions to their initial data in $L_x^2(\mathbb{R}^n)$. Cowling \cite{Cow} proved that \eqref{eq:1.3} holds for all $q = 2$ and $s > \frac{1}{2}$. Rogers and Villarroya \cite{RV} later extended this result to all $q \in (2, \infty]$. For endpoint estimates and historical remarks, see Cho, Lee, and Li \cite{CLL}. When the initial data lies in Sobolev spaces $W_x^{s,p}(\mathbb{R}^n)$, and $q = p$, \eqref{eq:1.3} holds for $p \geq \frac{2(n+1)}{n-1}$ and $s > (n-1)\left| \frac{1}{2} - \frac{1}{p} \right|$, as a corollary of Bourgain and Demeter \cite{BoD}. The two-dimensional case of this estimate can be found in Guth, Wang, and Zhang \cite{GWZ}, while an endpoint result in dimensions $n\geq 4$ was established by Heo, Nazarov, and Seeger \cite{HNS}. When $r = 2$ and $q = p > \frac{2(n+2)}{n}$, \eqref{eq:1.3} is closely related to Stein's square function estimates for Bochner-Riesz operators, as discussed in Lee, Rogers, and Seeger \cite{LRS12}. Additionally, analogous results for Fourier integral operators can be found in \cite{Sog, MSS, MSS92, GLMX}, and for compact manifolds in Beltran, Hickman, and Sogge \cite{B}.

It is also natural to consider the time-space estimate for $T$, commonly referred to as the Strichartz estimate. For the half-wave propagator, this estimate can be expressed as
\begin{align}\label{eq:1.5}
\left\| e^{iu\sqrt{-\Delta}} f \right\|_{L^r_u L_x^q(\mathbb{R}^n \times \mathbb{R})} \lesssim \| f \|_{\dot{H}_x^s(\mathbb{R}^n)},
\end{align}
where it is well-known that \eqref{eq:1.5} holds if $q, r \geq 2$, $\frac{1}{r} + \frac{n}{q} = \frac{n}{2} - s$, and $\frac{2}{r} + \frac{n-1}{q} \leq \frac{n-1}{2}$, with the exception of $(r, q) = (2, \infty)$ when $n = 3$. For further details and historical remarks, see Keel and Tao \cite{KT}. When $q = \infty$, \eqref{eq:1.5} is also connected to a Kakeya-type problem for circles, particularly when the initial data belongs to $W_x^{\frac{1}{2}, 3}(\mathbb{R}^2)$; see Wolff \cite{Wolff}. Over the past few decades, extensive literature has been devoted to this problem, including \cite{Str, LRV08, LV08, BBGL}.

We now state our second result concerning the time-space estimate for $T$.

\begin{theorem}\label{thm3}
Let $\gamma$ be as defined in Theorem \ref{thm1}. Then, there exists a positive constant $C$  depending only on $\gamma$, $p$, $q$ and $r$, such that for all $f \in L_x^p(\mathbb{R}^2)$,
\begin{align}\label{eq:1.7}
\left\| Tf \right\|_{L^r_u L_x^q(\mathbb{R}^2 \times [1,2])} \leq C \| f \|_{L_x^p(\mathbb{R}^2)},
\end{align}
provided that $(\frac{1}{p}, \frac{1}{q})$ satisfies $1 + (1 + \omega)(\frac{1}{q} - \frac{1}{p}) > 0$ and one of the following conditions holds:
\begin{enumerate}
 \item[$A:$] For $1 \leq p \leq q \leq r \leq \infty$, $(\frac{1}{p}, \frac{1}{q})$ satisfies $\tilde{s}_{A,0}(p, q, r) < 0$;
 \item[$B:$] For $1 \leq p \leq r \leq q \leq \infty$, $(\frac{1}{p}, \frac{1}{q})$ satisfies $\tilde{s}_{B,0}(p, q, r) < 0$, where
   \[
   \tilde{s}_{B,0}(p, q, r) := \min\left\{ \tilde{s}_{A,0}(p, r, r) + 2\left(\frac{1}{r} - \frac{1}{q}\right), \tilde{s}_{A,0}(p, q, q) \right\};
   \]
 \item[$C:$] For $1 \leq r < p \leq q \leq \infty$, $(\frac{1}{p}, \frac{1}{q})$ satisfies $\tilde{s}_{C,0}(p, q, r) < 0$, where
   \[
   \tilde{s}_{C,0}(p, q, r) := \min\left\{ \tilde{s}_{A,0}(p, p, p) + 2\left(\frac{1}{p} - \frac{1}{q}\right), \tilde{s}_{A,0}(p, q, q) \right\}.
   \]
\end{enumerate}
Conversely, \eqref{eq:1.7} holds only if $(\frac{1}{p}, \frac{1}{q})$ satisfies the following conditions:
\begin{enumerate}
  \item[\rm(I)] $\frac{1}{q} \leq \frac{1}{p}$;
  \item[\rm(II)] $1 + (1 + \omega)(\frac{1}{q} - \frac{1}{p}) \geq 0$;
  \item[\rm(III)] $\frac{1}{q} \geq \max\left\{ \frac{2}{p} - 1, \frac{1}{p} - \frac{1}{3} \right\}$;
  \item[\rm(IV)] $\frac{1}{q} \geq \max\left\{ \frac{1}{2p} - \frac{1}{2r}, \frac{2}{p} - \frac{r+1}{r} \right\}$ and $\max\left\{ \frac{1}{r} - \frac{1}{q}, 0 \right\} \geq \tilde{s}_{A,0}(p, q, q)$ for all $r \in [1, \infty]$.
\end{enumerate}
Here, for any $1 \leq p, q, r \leq \infty$, we define
\begin{equation*}
\tilde{s}_{A,0}(p, q, r) :=
\begin{cases}
\frac{1}{p} - \frac{2}{q} - \frac{1}{r}, & \text{for } q \geq 3p'; \\
\frac{1}{p} - \frac{2}{q} + \left( \frac{1}{2p} + \frac{1}{2q} - \frac{1}{2} \right) \frac{q}{r}, & \text{for } p' < q < 3p'; \\
\frac{2}{p} - \frac{1}{q} - 1, & \text{for } q \leq p'.
\end{cases}
\end{equation*}
\end{theorem}

\begin{remark}
From the definition of $\tilde{s}_{A,0}$ with $p' < q < 3p'$, it is evident that Theorem \ref{thm3} is not always sharp, since the indices $\frac{1}{p}$, $\frac{1}{q}$ and $\frac{1}{r}$ in $\tilde{s}_{A,0}$ is not a linear combination. However, by interpolating between the vertices of the range of $(\frac{1}{p}, \frac{1}{q}, \frac{1}{r})$, we may obtain some sharp results. It is natural to investigate what occurs in Theorem \ref{thm3} when $p = 2$. By interpolating between the points $(\frac{1}{q}, \frac{1}{r})=(\frac{1}{6}, \frac{1}{6})$ and $(\frac{1}{4}, 0)$, we obtain that
\begin{align}\label{eq:1.8}
\left\|Tf\right\|_{L^r_uL_x^q(\mathbb{R}^2\times [1,2])} \lesssim \|f\|_{L_x^{2}(\mathbb{R}^{2})}
\end{align}
holds if $(\frac{1}{q}, \frac{1}{r})$ lies on the closed line segment $[(\frac{1}{2}, 1), (\frac{1}{2}, 0)]$ or within the open pentagon with vertices $(\frac{1}{6}, \frac{1}{6})$, $(\frac{1}{2}, 1)$, $(\frac{1}{2}, 0)$, and $(\frac{1}{4}, 0)$. Moreover, this estimate is sharp except the boundary lines. The range of $(\frac{1}{q}, \frac{1}{r})$ is illustrated in Figure \ref{Figure:2}.

\begin{figure}[htbp]
  \centering
  \includegraphics[width=3.5in]{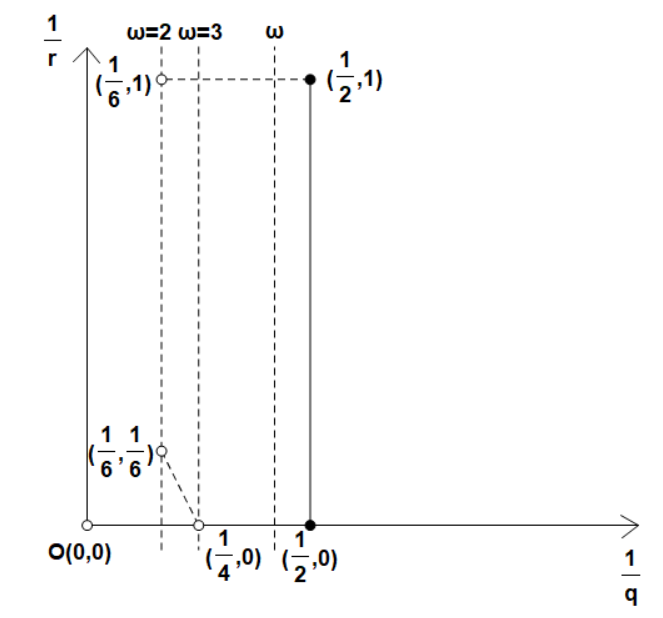}
  \caption{The range of $(\frac{1}{q}, \frac{1}{r})$ in Theorem \ref{thm3} when $p = 2$.}
  \label{Figure:2}
\end{figure}

Furthermore, a direct computation reveals that Theorem \ref{thm3} is almost sharp for all $p \in [1, \frac{3}{2}] \cup [4, \infty]$. The corresponding range of $(\frac{1}{q}, \frac{1}{r})$ can be found in Figure \ref{Figure:3}.

\begin{figure}[htbp]
  \centering
  \includegraphics[width=5in]{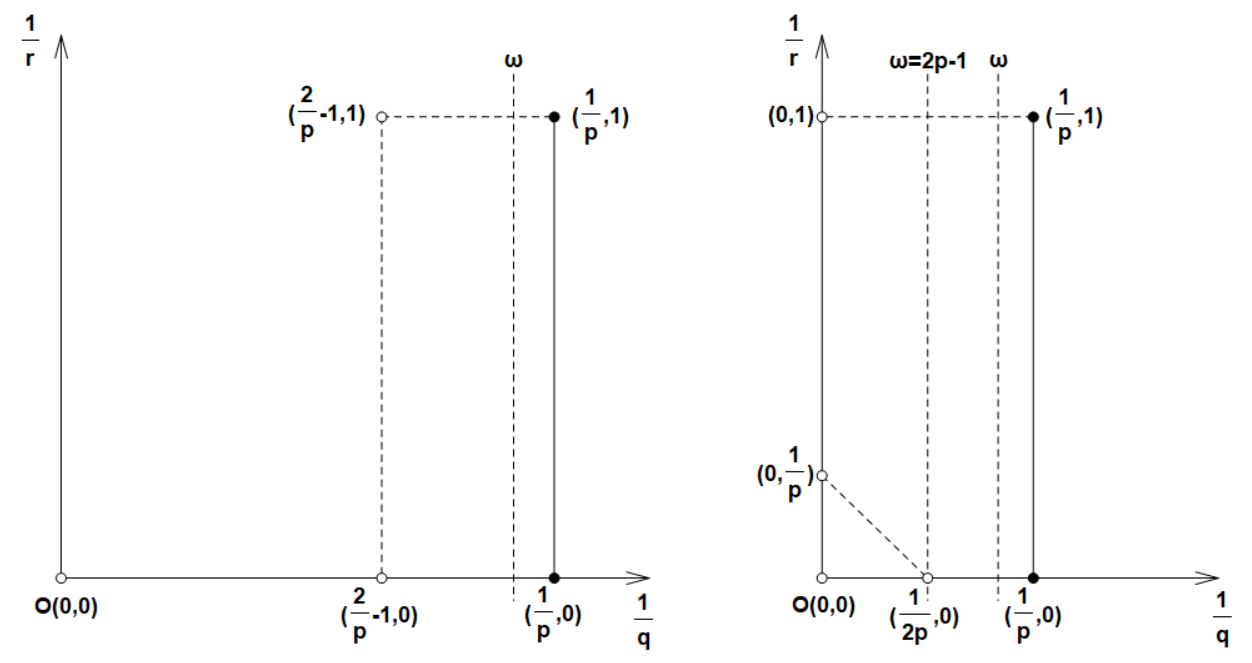}
  \caption{The range of $(\frac{1}{q}, \frac{1}{r})$ in Theorem \ref{thm3} for $p \in [1, \frac{3}{2}]$ and $p \in [4, \infty]$.}
  \label{Figure:3}
\end{figure}
\end{remark}

\begin{remark}
We present the detailed result in Theorem \ref{thm3} with $r = 2$. Combining Theorem \ref{thm3} with Remark \ref{remark5.1}, if $\left( \frac{1}{p}, \frac{1}{q} \right)$ lies on the open line segment $((0,0), (1,1))$ or within the open (dashed) pentagon with vertices $(0,0)$, $(1,1)$, $\left( \frac{2}{3}, \frac{1}{3} \right)$, $\left( \frac{1}{2}, \frac{1}{6} \right)$, and $\left( \frac{1}{4}, 0 \right)$, we have
\begin{align}\label{eq:1.9}
\left\|Tf\right\|_{L^2_uL_x^q(\mathbb{R}^2\times [1,2])} \lesssim \|f\|_{L_x^{p}(\mathbb{R}^{2})}.
\end{align}
Conversely, \eqref{eq:1.9} holds only if $\left( \frac{1}{p}, \frac{1}{q} \right)$ belongs to the closed (black) quadrilateral with vertices $(0,0)$, $(1,1)$, $\left( \frac{2}{3}, \frac{1}{3} \right)$, and $\left( \frac{1}{3}, 0 \right)$, see Figure \ref{Figure:4}. We note that if $\gamma(t) := t^d$ with $d \geq 3$, the time-space estimate $T : L_x^p(\mathbb{R}^2) \rightarrow L^{2}_uL_x^q(\mathbb{R}^2 \times \mathbb{R})$ established in Theorem \ref{thm3} is essentially sharp. However, if $\gamma(t) := t^d$ with $d \in (0,1) \cup (1,2)$, the condition $1 + (1 + \omega) \left( \frac{1}{q} - \frac{1}{p} \right) > 0$ is redundant.
\begin{figure}[htbp]
  \centering
  \includegraphics[width=3.8in]{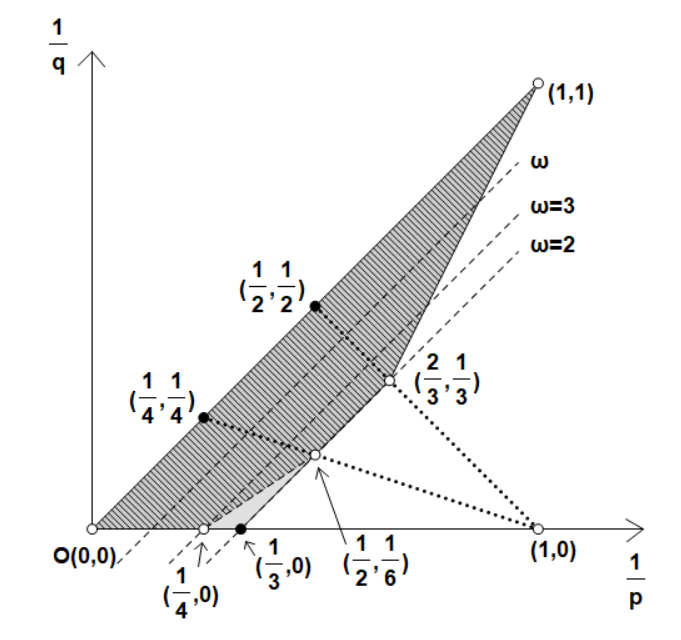}
  \caption{The range of $\left( \frac{1}{p}, \frac{1}{q} \right)$ in Theorem \ref{thm3} with $r = 2$.}
  \label{Figure:4}
\end{figure}
\end{remark}

The main contribution of this paper is to establish the optimal range of $\left( \frac{1}{p}, \frac{1}{q}, \frac{1}{r} \right)$ for the space-time estimate of $T$, as stated in Theorem \ref{thm1}. Indeed, the range of $\left( \frac{1}{p}, \frac{1}{q} \right)$ for $T : L_x^p(\mathbb{R}^2) \rightarrow L_x^qL^{\infty}_u(\mathbb{R}^2 \times [1,2])$ determines all other ranges of $\left( \frac{1}{p}, \frac{1}{q}, \frac{1}{r} \right)$. To achieve this result we need the sharp estimate for
\begin{align*}
F_{j,k}f(x,u) := \Omega(x,u) \int_{\mathbb{R}^2} e^{ix \cdot \xi} e^{i \varphi_j(u,\xi,t_0)} m_{j,k}(u,\xi) \hat{f}(\xi) \, d\xi,
\end{align*}
defined in \eqref{eq:2.10}, which is essentially a Fourier integral operator with phase function $x \cdot \xi - u |\xi|$. Thus it is basically a half-wave propagator. These estimates for $F_{j,k}$ generalize the results of Mockenhaupt, Seeger, and Sogge \cite[Theorem 3.1]{MSS}, and Lee \cite[Corollary 1.5]{Lee06} when we restrict our attention to $P_k$.

It is evident that the time-space estimate for $T$ differs from the corresponding space-time estimate. Firstly, the latter is translation-invariant when the temporal integral is evaluated before the spatial integral, but not the former. Secondly, as stated in Subsection 4.1, the time-space estimate for $T$ not always can be reduced to the corresponding time-space estimate similar to \eqref{eq:1.2}. Thirdly, we can only obtain a common estimate $$
\|F_{j,k}f\|_{L^{r}_uL_x^{\infty}(\mathbb{R}^2\times \mathbb{R}) } \lesssim2^k\|f\|_{L_x^{1}(\mathbb{R}^{2})}$$ for all $r\in[1, \infty]$, not as the corresponding space-time estimate can be improved directly when $r=1$; see \eqref{eq:3.3}. Our main tools will be interpolation, a fixed time estimate for $F_{j,k}$ based on Seeger, Sogge and Stein \cite[Theorem 2.1]{SSS}, Gao, Liu, Miao and Xi \cite[Theorem 1.4]{GLMX} and some results stated in Stein \cite[Chapters VIII and IX]{Stein}.

To establish the necessity part in Theorem \ref{thm3}, we utilized the sharp range of $\left( \frac{1}{p}, \frac{1}{q} \right)$ in $T : L_x^p(\mathbb{R}^2) \rightarrow L_{x,u}^q(\mathbb{R}^2 \times [1,2])$, not only the sharp range of $(\frac{1}{p}, \frac{1}{q})$ in $T : L_x^p(\mathbb{R}^2) \rightarrow L^{\infty}_uL_x^q(\mathbb{R}^2 \times [1,2])$. We also constructed some examples inspired by previous papers.  However, it should be noted that Theorem \ref{thm3} may not be sharp in some case of $p \in \left( \frac{3}{2}, 2 \right) \cup (2, 4)$. To be more precise, it follows from Theorem \ref{thm3} that
\begin{align}\label{eq:1.10}
\left\|Tf\right\|_{L^p_uL_x^{\infty}(\mathbb{R}^2\times [1,2])} \lesssim \|f\|_{L_x^{p}(\mathbb{R}^{2})}
\end{align}
holds if $\max\{1+\omega, 4\}<p\leq\infty$ and only if $\max\{1+\omega, 3\}\leq p\leq\infty$. Thus the range of $(\frac{1}{p}, \frac{1}{q}, \frac{1}{r})$ in \eqref{eq:1.7} is not sharp. Moreover, we can conjecture that the infimum of $p$ in \eqref{eq:1.10} is $\max\{1+\omega, 3\}$ by the estimate $(2)$ established in Wolff \cite{Wolff}. From Theorem \ref{thm3}, we also note that \eqref{eq:1.10} is essentially sharp if $\gamma(t):=t^d$ with $d\geq 3$.

This paper is organized as follows: In Section 2, we provide preliminaries based on prior papers \cite{LiuYu, LLLY} and reduce our estimates to a Fourier integral operator $F_{j,k}$. In Section 3, we prove Theorem \ref{thm1} by addressing the sufficiency and necessity in Subsections 3.1 and 3.2, respectively. In Section 4, we discuss the time-space estimate for $T$, i.e., Theorem \ref{thm3}. Specifically, Subsection 4.1 elaborates on the estimate \eqref{eq:1.7} by analyzing three cases: $1 \leq p \leq q \leq r \leq \infty$, $1 \leq p \leq r \leq q \leq \infty$, and $1 \leq r < p \leq q \leq \infty$. The necessary conditions are verified in Subsection 4.2.

Finally, we establish some notation conventions. Throughout this paper, the letter $C$ denotes a positive constant, independent of the essential variables, but whose value may vary from line to line. We use $C_{(u, N, \alpha, \ldots)}$ to denote a positive constant depending on the parameters $u, N, \alpha, \ldots$. The notation $a \lesssim b$ (or $a \gtrsim b$) indicates that there exists a finite positive constant $C$ such that $a \leq Cb$ (or $a \geq Cb$). $a \approx b$ means both $a \lesssim b$ and $b \lesssim a$ hold. For $x \in \mathbb{R}^n$ and $r \in (0, \infty)$, $B(x, r) := \{ y \in \mathbb{R}^n : |x - y| < r \}$ and $B^{\complement}(x, r)$ is its complement in $\mathbb{R}^n$. Let $\mathbb{S}^{n-1}$ denote the unit sphere on $\mathbb{R}^n$. $\mathcal{S}(\mathbb{R}^n)$ denotes the collection of all Schwartz functions on $\mathbb{R}^n$. $\hat{f}$ and $f^{\vee}$ denote the Fourier transform and inverse Fourier transform of $f$, respectively. For $1 < q \leq \infty$, $q'$ denotes the adjoint number of $q$, i.e., $\frac{1}{q} + \frac{1}{q'} = 1$. $\mathbb{Z}_0^{-} := \{ 0, -1, -2, \ldots \}$, $\mathbb{N} := \{ 1, 2, \ldots \}$, $\mathbb{N}_0 := \{ 0, 1, 2, \ldots \}$, and $\mathbb{R}^{+} := (0, \infty)$. For any set $E$, $\chi_E$ denotes the characteristic function of $E$.

\section{Preliminaries}\label{preliminaries}

The proofs of Theorems \ref{thm1} and \ref{thm3} rely on some results established in previous papers \cite{LiuYu, LLLY}. For ease of reading, we provide a sketch of reducing the estimate to a Fourier integral operator in this section. Detailed discussions can be found in \cite{LiuYu, LLLY}. We begin by collecting some lemmas.

\begin{lemma}\label{lemma 2.1}
(\cite[Lemma 2.2]{LSY} or \cite[Lemma 2.2]{LYu}) Let $\gamma$ be defined as in Theorem \ref{thm1} and let $\Gamma_{j}(t) := \frac{\gamma(2^jt)}{\gamma(2^j)}$ with $j \in \mathbb{Z}_0^{-}$ and $t \in [\frac{1}{2}, 2]$. The following inequalities hold uniformly in $j$:
\begin{enumerate}
  \item[\rm(i)] $e^{-C^{(1)}_{2}} \leq \Gamma_{j}(t) \leq e^{C^{(1)}_{2}}$;
  \item[\rm(ii)] $\frac{C^{(1)}_1}{2e^{C^{(1)}_{2}}} \leq |\Gamma_{j}'(t)| \leq 2e^{C^{(1)}_{2}}C^{(1)}_2$;
  \item[\rm(iii)] $\frac{C^{(2)}_1}{4e^{C^{(1)}_{2}}} \leq |\Gamma_{j}''(t)| \leq 4e^{C^{(1)}_{2}}C^{(2)}_{2}$;
  \item[\rm(iv)] $|\Gamma_{j}^{(k)}(t)| \leq 2^{k} e^{C^{(1)}_{2}}C^{(k)}_{2}$ for all $2 \leq k \leq N$ and $k \in \mathbb{N}$;
  \item[\rm(v)] $|((\Gamma_{j}')^{-1})^{(k)}(t)| \lesssim 1$ for all $0 \leq k < N$ and $k \in \mathbb{N}$, where $(\Gamma_{j}')^{-1}$ is the inverse function of $\Gamma_{j}'$.
\end{enumerate}
\end{lemma}

\begin{lemma}\label{lemma 2.2}
(\cite[Lemma 2.4.2]{So}) Suppose that $F$ is $C^1(\mathbb{R})$. Then, if $q > 1$,
\begin{align*}
\sup_{u \in [1,2]} |F(u)|^q \leq |F(1)|^q + q \left( \int_1^2 |F(u)|^q \, \mathrm{d}u \right)^{\frac{1}{q'}} \left( \int_1^2 |F'(u)|^q \, \mathrm{d}u \right)^{\frac{1}{q}}.
\end{align*}
\end{lemma}

\begin{lemma}\label{lemma 2.3}
(\cite[Lemma 2.3]{LiuYu}) Recall that $\omega = \limsup_{t \rightarrow 0^{+}} \frac{\ln|\gamma(t)|}{\ln t}$. Then, for all $(\frac{1}{p}, \frac{1}{q})$ satisfying $\frac{1}{q} \leq \frac{1}{p}$ and $1 + (1 + \omega)(\frac{1}{q} - \frac{1}{p}) > 0$, we have
\begin{align*}
\sum_{j \in \mathbb{Z}_0^{-}} 2^j |2^j \gamma(2^j)|^{\frac{1}{q} - \frac{1}{p}} < \infty.
\end{align*}
\end{lemma}


We consider the space-time estimate for the operator
\begin{align}\label{eq:2.1}
Tf(x,u) = \int_{0}^{1} f(x_1 - ut, x_2 - u \gamma(t)) \, \mathrm{d}t,
\end{align}
where, without loss of generality, we assume \( f \geq 0 \). Let \( \psi \) be a smooth and positive function supported on the interval \( \{ t \in \R : \frac{1}{2} \leq t \leq 2 \} \) such that $\sum_{j \in \Z} \psi_j(t) = 1$ for all  $t > 0$,
where \( \psi_j(t) := \psi(2^{-j}t) \). For any $j\in \mathbb{Z}$, we define
$\delta_jf(x_1,x_2):=f(2^jx_1,\gamma(2^j)x_2)$. Note that
$|2^j\gamma(2^j)|^{\frac{1}{q}}\|\delta_jf\|_{L^{q}(\mathbb{R}^{2})}=\|f\|_{L^{q}(\mathbb{R}^{2})}$,
it suffices to show that
\begin{align*}
\sum_{j \in \Z_0^{-}} 2^j |2^j \gamma(2^j)|^{\frac{1}{q} - \frac{1}{p}} \left\| \widetilde{T}_{j} \right\|_{L_x^{p}(\R^{2}) \rightarrow L_x^q L^r_u(\R^2 \times [1,2])} \lesssim 1,
\end{align*}
where
\begin{align*}
\widetilde{T}_{j}f(x,u) := \int_{0}^{\infty} f(x_1 - ut, x_2 - u \Gamma_j(t)) \psi(t) \, \mathrm{d}t,
\end{align*}
and the definition of \( \Gamma_j \) can be found in Lemma \ref{lemma 2.1}. Consequently, by Lemma \ref{lemma 2.3}, the proof reduces to establishing the estimate
$\widetilde{T}_{j} : L_x^p(\R^2) \rightarrow L_x^q L^r_u(\R^2 \times [1,2])$
uniformly in \( j \in \Z_0^{-} \).

Let \( \Psi^0 := \sum_{k \leq 0} \psi_k \) and define
\begin{align*}
H_j(u, \xi) := \int_{\R} e^{-iu \xi_1 t - i u \xi_2 \Gamma_j(t)} \psi(t) \, \mathrm{d}t,
\end{align*}
where \( \xi := (\xi_1, \xi_2) \). We decompose \( \widetilde{T}_{j} \) as
\begin{align*}
\widetilde{T}_{j}f(x,u) = \widetilde{T}^0_{j}f(x,u) + \sum_{k \geq 1} \widetilde{T}_{j,k}f(x,u),
\end{align*}
where the multipliers of \( \widetilde{T}^0_{j} \) and \( \widetilde{T}_{j,k} \) are \( H_j(u, \xi) \Psi^0(u|\xi|) \) and \( H_j(u, \xi) \psi_k(u|\xi|) \), respectively. For \( \widetilde{T}^0_{j} \), by Lemma \ref{lemma 2.1}, we can obtain that $|\partial_u^{\alpha}\partial^{\beta}_{\xi}(H_j(u, \xi)\Psi^0(u|\xi|))|\lesssim (1+|\xi|)^{-3}$ for all $(\alpha,\beta)\in \mathbb{N}_0\times \mathbb{N}^2_0$ with $|\alpha|\leq 1$ and $|\beta|\leq 3$, where the implicit constant is independent of $u\in[1,2]$ and \( j \in \Z_0^{-} \). Furthermore, by Lemma \ref{lemma 2.2}, H\"older's inequality, and Young's inequality, it follows that
\begin{align}\label{eq:2.4}
\left\| \widetilde{T}^0_{j}f \right\|_{L_x^q L^r_u(\R^2 \times [1,2])} \leq \left\| \widetilde{T}^0_{j}f \right\|_{L_x^q L^{\infty}_u(\R^2 \times [1,2])} \lesssim \| f \|_{L_x^{p}(\R^{2})}
\end{align}
for all \( 1 \leq r \leq \infty \) and \( 1 \leq p \leq q \leq \infty \).

For \( \widetilde{T}_{j,k} \), we further split it into \( \widetilde{T}^1_{j,k} \) and \( \widetilde{T}^2_{j,k} \), where the multipliers of \( \widetilde{T}^1_{j,k} \) and \( \widetilde{T}^2_{j,k} \) are
\[
\rho\left( \frac{|\xi_1|}{|\xi_2|} \right) H_j(u, \xi) \psi_k(u|\xi|)
~~\textrm{and}~~
\left(1 - \rho\left( \frac{|\xi_1|}{|\xi_2|} \right)\right) H_j(u, \xi) \psi_k(u|\xi|),
\]
respectively. Here, \( \rho \in C^{\infty}_{c}(\R^{+}) \) is a smooth cutoff function satisfying \( \rho = 1 \) on the interval
$
[ \frac{C^{(1)}_{1}}{10e^{C^{(1)}_{2}}} , 10 e^{C^{(1)}_{2}} C^{(1)}_2 ].
$
Then, as in \eqref{eq:2.4}, by Lemmas \ref{lemma 2.1} and \ref{lemma 2.2}, we obtain the decay estimate
\begin{align*}
\left\| \widetilde{T}^2_{j,k}f \right\|_{L_x^q L^r_u(\R^2 \times [1,2])} \lesssim 2^{-\frac{k}{q'}} \| f \|_{L_x^{p}(\R^{2})}
\end{align*}
for all \( 1 \leq r \leq \infty \) and \( 1 \leq p \leq q \leq \infty \). Therefore, it suffices to analyze the operator \( \widetilde{T}^1_{j,k} \).

Let the phase function in \( H_j(u, \xi) \) be
$\varphi_j(u,\xi,t) := -u \xi_1 t - u \xi_2 \Gamma_j(t)$.
We focus on the case where there exists a critical point \( t_0 \in \left( \frac{1}{2}, 2 \right) \) of \( \varphi_j \). By the method of stationary phase, we write
\begin{align}\label{eq:2.5}
\widetilde{T}^1_{j,k}f(x,u) = \int_{\R^{2}} e^{ix \cdot \xi} e^{i \varphi_j(u,\xi,t_0)} m_{j,k}(u,\xi) \hat{f}(\xi) \, \mathrm{d}\xi,
\end{align}
where
\begin{align*}
m_{j,k}(u,\xi) := \rho\left( \frac{|\xi_1|}{|\xi_2|} \right) \psi_k(u|\xi|) \int_{\R} e^{-iu \xi_2 t^2 \eta_j(t,t_0)} \psi(t + t_0) \, \mathrm{d}t,
\end{align*}
with
$\eta_j(t,t_0) := \frac{\Gamma''_j(t_0)}{2} + \frac{t}{2} \int_0^1 (1 - \theta)^2 \Gamma'''_j(\theta t + t_0) \, \mathrm{d}\theta$.
Moreover, for all \( (\alpha,\beta) \in \N^2_0 \times \N_0 \) with \( |\alpha| + |\beta| < N \),
\begin{align}\label{eq:2.6}
\left| \partial_{u}^{\beta} \partial_{\xi}^{\alpha} m_{j,k}(u,\xi) \right| \lesssim (1 + |\xi|)^{-\frac{1}{2} - |\alpha|}
\end{align}
uniformly in \( u \in ( \frac{1}{2}, \frac{5}{2} ) \).

We can also express \( \widetilde{T}^1_{j,k} \) as
\begin{align*}
\widetilde{T}^1_{j,k}f(x,u) = \int_{\R^{2}} \widetilde{K}_{j,k}(x - y, u) f(y) \, \mathrm{d}y,
\end{align*}
where the kernel is given by
\begin{align}\label{eq:2.7}
\widetilde{K}_{j,k}(x,u) := \int_{\R^{2}} e^{ix \cdot \xi} e^{i \varphi_j(u,\xi,t_0)} m_{j,k}(u,\xi) \, \mathrm{d}\xi.
\end{align}
Furthermore, we split \( \widetilde{T}^1_{j,k} \) into the sum of \( \widetilde{T}^{1,a}_{j,k} \) and \( \widetilde{T}^{1,b}_{j,k} \), whose kernels are
$
\chi_{B^{\complement}(0,\varpi)} (x) \widetilde{K}_{j,k}(x,u)
$
and
$
\chi_{B(0,\varpi)} (x) \widetilde{K}_{j,k}(x,u),
$
respectively. Here, \( \varpi \) is a positive constant chosen large enough such that
\[
\left| \nabla_{\xi}[(x - y) \cdot \xi + \varphi_j(u,\xi,t_0)] \right| \gtrsim |x - y|
\]
when \( |x - y| \geq \varpi \). The desired decay estimate for \( \widetilde{T}^{1,a}_{j,k} \) can be obtained via integration by parts. It then suffices to consider \( \widetilde{T}^{1,b}_{j,k} \).

Let \( \Omega: \R^2 \times \R \rightarrow \R \) be a nonnegative smooth function such that \( \Omega(x,u) = 1 \) on
\[
\{(x,u) \in \R^2 \times \R : |x| \leq \varpi, ~ u \in [1,2]\}
\]
and vanishes outside
\[
\{(x,u) \in \R^2 \times \R : |x| \leq 2\varpi, ~ u \in (\tfrac{1}{2}, \tfrac{5}{2})\}.
\]
Define
\begin{align}\label{eq:2.8}
F_{j,k}f(x,u) := \Omega(x,u) \int_{\R^2} \widetilde{K}_{j,k}(x - y, u) f(y) \, \mathrm{d}y
\end{align}
and
\begin{align}\label{eq:2.08}
E_{j,k}f(x,u) := \Omega(x,u) \int_{\R^2} \chi_{B(0,\varpi)} (x - y) \widetilde{K}_{j,k}(x - y, u) f(y) \, \mathrm{d}y.
\end{align}
To estimate \( \widetilde{T}^{1,b}_{j,k} \), as in \cite{LiuYu}, it suffices to obtain a corresponding decay space-time estimate for \( F_{j,k} \). Indeed, we partition \( \R^2 \) into \( \bigcup_{i \in \Z^2} B(x_i, \varpi) \) such that \( |x_i - x_{i'}| \approx |i - i'| \varpi \) for any \( i \neq i' \). Then, \( \|\widetilde{T}^{1,b}_{j,k}f\|^q_{L_x^q L^r_u(\R^2 \times [1,2])} \) is bounded from above by
\begin{align}\label{eq:2.9}
\sum_{i \in \Z^2} \left\| \left(E_{j,k} - F_{j,k}\right) \left(f(\cdot - x_i)\right) \right\|^q_{L_x^q L^r_u(\R^2 \times \R)} + \sum_{i \in \Z^2} \left\| F_{j,k} \left(f(\cdot - x_i)\right) \right\|^q_{L_x^q L^r_u(\R^2 \times \R)}.
\end{align}
The first term can be handled by a similar argument as for \( \widetilde{T}^{1,a}_{j,k} \), and the second term reduces to a corresponding decay space-time estimate for \( F_{j,k} \).

Putting things together, we have reduced our estimate to \( F_{j,k} \). We note that
\begin{align}\label{eq:2.10}
F_{j,k}f(x,u) := \Omega(x,u) \int_{\R^2} e^{ix \cdot \xi} e^{i \varphi_j(u,\xi,t_0)} m_{j,k}(u,\xi) \hat{f}(\xi) \, \mathrm{d}\xi,
\end{align}
which is a Fourier integral operator of order \( -\frac{1}{2} \) and satisfies the special assumption that
\begin{align}\label{eq:2.100}
\text{rank}~ \partial^2_{\xi\xi} \left(x \cdot \xi + \varphi_j(u,\xi,t_0)\right) = 1
\end{align}
when \( \xi \neq 0 \).



\section{Proof of Theorem \ref{thm1}}

In this section, we present the proof of Theorem \ref{thm1}, building upon the preliminary results established in Section \ref{preliminaries}. The proof is divided into two parts.

\subsection{Sufficiency of the region $(\frac{1}{p}, \frac{1}{q})$  for all $r \in [1, \infty]$}\label{subsect:3.1}

We begin the proof with the following two lemmas.

\begin{lemma}\label{lemma3.1}
For $\widetilde{K}_{j,k}$ defined in \eqref{eq:2.7} and $N \in \mathbb{N}$ sufficiently large, for all $\alpha \in \mathbb{N}_0$,
there exists a positive constant $C_{(N, \alpha)}$, depending on $N$ and $\alpha$, such that
\begin{align*}
\left|\partial_u^{\alpha} \widetilde{K}_{j,k}(x, u)\right| \leq C_{(N, \alpha)} \frac{2^{(1 + \alpha)k}}{\left(1 + 2^k \left||x| - u\right|\right)^N}
\end{align*}
holds for all $x \in \mathbb{R}^2$ and $u \in (\frac{1}{2}, \frac{5}{2})$.
\end{lemma}

\begin{proof}
Note that $t_0 = (\Gamma'_j)^{-1}(-\frac{\xi_1}{\xi_2}) \in (\frac{1}{2}, 2)$. By a change of variables, we can express
\begin{align*}
\widetilde{K}_{j,k}(x, u) = 2^{2k} \int_{\mathbb{R}^2} e^{i2^k x \cdot \xi} e^{i2^k \varphi_j(u, \xi, t_0)} m_{j,k}(u, 2^k \xi) \, \mathrm{d}\xi.
\end{align*}
From $\partial_{\xi_1} t_0 = -\frac{1}{\xi_2} \frac{1}{\Gamma''_j(t_0)}$ and $\partial_{\xi_2} t_0 = \frac{\xi_1}{\xi_2^2} \frac{1}{\Gamma''_j(t_0)}$, a straightforward calculation yields
\begin{align*}
\left\{
\begin{aligned}
&\partial_{\xi_1 \xi_1} (x \cdot \xi + \varphi_j(u, \xi, t_0)) = u \frac{1}{\xi_2} \frac{1}{\Gamma''_j(t_0)}; \\
&\partial_{\xi_1 \xi_2} (x \cdot \xi + \varphi_j(u, \xi, t_0)) = -u \frac{\xi_1}{\xi_2^2} \frac{1}{\Gamma''_j(t_0)}; \\
&\partial_{\xi_2 \xi_1} (x \cdot \xi + \varphi_j(u, \xi, t_0)) = -u \frac{\xi_1}{\xi_2^2} \frac{1}{\Gamma''_j(t_0)}; \\
&\partial_{\xi_2 \xi_2} (x \cdot \xi + \varphi_j(u, \xi, t_0)) = u \frac{\xi_1^2}{\xi_2^3} \frac{1}{\Gamma''_j(t_0)}.
\end{aligned}
\right.
\end{align*}
Combining this with Lemma \ref{lemma 2.1} and the facts that $|\xi| \approx 1$ and $|\xi_1| \approx |\xi_2|$, we obtain
\begin{align*}
\left|\partial^{\beta}_{\xi} (x \cdot \xi + \varphi_j(u, \xi, t_0))\right| \gtrsim 1
\end{align*}
for all multi-indices $\beta$ with $|\beta| = 2$. Noting that $\partial_{\xi} (x \cdot \xi + \varphi_j(u, \xi, t_0)) = x - u(t_0, \Gamma_j(t_0))$ and $|(t_0, \Gamma_j(t_0))| \approx 1$, integration by parts (using Stein \cite[Chapter VIII, Section 2, Proposition 5]{Stein} and \eqref{eq:2.6}) yields
\begin{align*}
\left|\widetilde{K}_{j,k}(x, u)\right| \lesssim \frac{2^k}{\left(1 + 2^k \left||x| - u\right|\right)^N}.
\end{align*}
A simple calculation provides the corresponding estimate for $|\partial_u^{\alpha} \widetilde{K}_{j,k}(x, u)|$. This completes the proof of the lemma.
\end{proof}

\begin{lemma}\label{lemma3.2}
Suppose \(G(x, u) \in C_c^1(\frac{1}{2}, \frac{5}{2})\) with respect to \(u\) for all \(x \in \mathbb{R}^2\). Then, for any \(1 \leq q, r \leq \infty\),
\begin{align*}
\left\|G\right\|_{L_x^q L^\infty_u(\mathbb{R}^2 \times \mathbb{R})} \lesssim \left\|G\right\|^{\frac{1}{r'}}_{L_x^q L^r_u(\mathbb{R}^2 \times \mathbb{R})} \left\|\partial_u G\right\|^{\frac{1}{r}}_{L_x^q L^r_u(\mathbb{R}^2 \times \mathbb{R})}.
\end{align*}
\end{lemma}

\begin{proof}
For \(r = \infty\), the result is trivial. For \(r \in [1, \infty)\), note that
\begin{align*}
|G(x, u)|^r \lesssim \int_\mathbb{R} |G(x, v)|^{r-1} |\partial_v G(x, v)| \, \mathrm{d}v
\end{align*}
holds for all \(u \in (\frac{1}{2}, \frac{5}{2})\) and \(x \in \mathbb{R}^2\). By H\"older's inequality,
\begin{align*}
\left\|G(x, u)\right\|_{L^\infty_u(\mathbb{R})} \lesssim \left(\left\|G(x, u)\right\|^{r-1}_{L^r_u(\mathbb{R})} \left\|\partial_u G(x, u)\right\|_{L^r_u(\mathbb{R})}\right)^{\frac{1}{r}}
\end{align*}
for all \(x \in \mathbb{R}^2\). Applying H\"older's inequality again yields the desired estimate. This completes the proof.
\end{proof}

We now proceed to prove the sufficiency part of Theorem \ref{thm1}. From Lemma \ref{lemma3.1}, it follows that $\|\widetilde{K}_{j,k}\|_{L_x^1(\mathbb{R}^2)} \lesssim 1$ uniformly in $u \in (\frac{1}{2}, \frac{5}{2})$. Combining this with Young's inequality, we obtain
\begin{align}\label{eq:3.1}
\left\|F_{j,k} f\right\|_{L_x^\infty L_u^\infty(\mathbb{R}^2 \times \mathbb{R})} \lesssim \|f\|_{L_x^\infty(\mathbb{R}^2)}.
\end{align}
Additionally, Lemma \ref{lemma3.1} implies
\begin{align}\label{eq:3.2}
\|F_{j,k} f\|_{L_x^\infty L_u^\infty(\mathbb{R}^2 \times \mathbb{R})} \lesssim 2^k \|f\|_{L_x^1(\mathbb{R}^2)}.
\end{align}
This bound can be refined by considering the $L^1$-norm in $u$. Specifically, we observe that
\begin{align*}
\int_{1}^2 \frac{2^k}{\left(1 + 2^k \left||x - y| - u\right|\right)^N} \, \mathrm{d}u \lesssim 1
\end{align*}
uniformly in $x, y \in \mathbb{R}^2$, which further implies
\begin{align}\label{eq:3.3}
\left\|F_{j,k} f\right\|_{L_x^\infty L_u^1(\mathbb{R}^2 \times \mathbb{R})} \lesssim \|f\|_{L_x^1(\mathbb{R}^2)}.
\end{align}
Similarly, by
\begin{align*}
\int_{\mathbb{R}^2} \frac{2^k}{\left(1 + 2^k \left||x - y| - u\right|\right)^N} \, \mathrm{d}x \lesssim 1
\end{align*}
uniformly in $y \in \mathbb{R}^2$ and $u \in (\frac{1}{2}, \frac{5}{2})$, we conclude that
\begin{align}\label{eq:3.4}
\left\|F_{j,k} f\right\|_{L_x^1 L_u^1(\mathbb{R}^2 \times \mathbb{R})} \lesssim \|f\|_{L_x^1(\mathbb{R}^2)}.
\end{align}

From \cite[(2.44)]{LiuYu}, for any $\epsilon \in (0, \infty)$, we have
\begin{align*}
\left\|F_{j,k} f\right\|_{L_x^5 L_u^5(\mathbb{R}^2 \times \mathbb{R})} \lesssim 2^{(\epsilon - \frac{1}{5})k} \|f\|_{L_x^{\frac{5}{2}}(\mathbb{R}^2)}.
\end{align*}
Furthermore, by Lemma \ref{lemma 2.2} or Lemma \ref{lemma3.2}, we find that
\begin{align}\label{eq:3.5}
\left\|F_{j,k} f\right\|_{L_x^5 L_u^\infty(\mathbb{R}^2 \times \mathbb{R})} \lesssim 2^{\epsilon k} \|f\|_{L_x^{\frac{5}{2}}(\mathbb{R}^2)}
\end{align}
holds for any $\epsilon \in (0, \infty)$. On the other hand, by \eqref{eq:2.6}, it is a consequence of Plancherel's theorem that
\begin{align}\label{eq:3.7}
\left\|F_{j,k} f\right\|_{L_x^2 L_u^2(\mathbb{R}^2 \times \mathbb{R})} \lesssim 2^{-\frac{1}{2}k} \|f\|_{L_x^2(\mathbb{R}^2)}.
\end{align}
Combining this with Lemma \ref{lemma 2.2} and \eqref{eq:2.6}, we obtain
\begin{equation}\label{eq:3.6}
\begin{split}
	\left\|F_{j,k} f\right\|_{L_x^2 L_u^\infty(\mathbb{R}^2 \times \mathbb{R})} &\lesssim \left\|F_{j,k} f(x, 1)\right\|_{L_x^2(\mathbb{R}^2)}  + \left(\left\|F_{j,k} f\right\|_{L_x^2 L_u^2(\mathbb{R}^2 \times \mathbb{R})} \left\|\partial_u (F_{j,k} f)\right\|_{L_x^2 L_u^2(\mathbb{R}^2 \times \mathbb{R})}\right)^{\frac{1}{2}} \\
&\lesssim \|f\|_{L_x^2(\mathbb{R}^2)}.
\end{split}
\end{equation}

From Figure \ref{Figure:1}, for all $r \in [1, \infty]$, it suffices to obtain a constant bound $C$ or a small bound $2^{\epsilon_1 k}$ at points $\{\text{O, A, B, C, D}\}$ for any $\epsilon_1 \in (0, \infty)$, and a decay bound $2^{-\epsilon_2 k}$ for $F_{j,k}$ at $(\frac{1}{4}, \frac{1}{4})$ with some $\epsilon_2 \in (0, \infty)$. More precisely, we first claim the following:

\begin{enumerate}
\item[$\bullet$] \textbf{Constant bound at point} O: For any $r \in [1, \infty]$, by H\"older's inequality and \eqref{eq:3.1}, it is easy to see that
\begin{align*}
\left\|F_{j,k} f\right\|_{L_x^\infty L_u^r(\mathbb{R}^2 \times \mathbb{R})} \lesssim \|f\|_{L_x^\infty(\mathbb{R}^2)}.
\end{align*}

\item[$\bullet$] \textbf{Constant bound at point} A: It follows from interpolation between the $ L_x^1(\mathbb{R}^2) \rightarrow L_x^1 L_u^1(\mathbb{R}^2 \times \mathbb{R})$ and $ L_x^2(\mathbb{R}^2) \rightarrow L_x^2 L_u^\infty(\mathbb{R}^2 \times \mathbb{R})$ estimates of $F_{j,k}$ established in \eqref{eq:3.4} and \eqref{eq:3.6}, respectively.

\item[$\bullet$] \textbf{Constant bound at point} B: For any $r\in[1, \infty]$, from \eqref{eq:3.6} and \eqref{eq:3.3}, reiterating this interpolation process, then the following estimate is valid:
\begin{align*}
\left\|F_{j,k} f\right\|_{L_x^{\frac{2r}{r-1}} L_u^r(\mathbb{R}^2 \times \mathbb{R})} \lesssim \|f\|_{L_x^{\frac{2r}{r+1}}(\mathbb{R}^2)}.
\end{align*}
This is the desired estimate at the point $\textrm{B}$.

\item[$\bullet$] \textbf{Small bound at point} C: For any $r\in[1, \infty]$, a further interpolation between \eqref{eq:3.5} and \eqref{eq:3.3} gives
\begin{align*}
\left\|F_{j,k} f\right\|_{L_x^{\frac{5r}{r-1}} L_u^r(\mathbb{R}^2 \times \mathbb{R})} \lesssim 2^{\frac{\epsilon}{r'}k} \|f\|_{L_x^{\frac{5r}{2r+3}}(\mathbb{R}^2)}
\end{align*}
for any $\epsilon \in (0, \infty)$.

\item[$\bullet$] \textbf{Constant bound at point} D: For any $r \in [1, \infty]$, interpolation between \eqref{eq:3.1} and \eqref{eq:3.3} yields
\begin{align*}
\left\|F_{j,k} f\right\|_{L_x^\infty L_u^r(\mathbb{R}^2 \times \mathbb{R})} \lesssim \|f\|_{L_x^r(\mathbb{R}^2)}.
\end{align*}
\end{enumerate}

Therefore, we reduce our proof to obtaining a decay bound at $(\frac{1}{4}, \frac{1}{4})$. For any $r \in [1, \infty]$, by H\"older's inequality, it is easy to see that $$\left\|F_{j,k} f\right\|_{L_x^4 L_u^{r}(\mathbb{R}^2 \times \mathbb{R})} \lesssim \left\|F_{j,k} f\right\|_{L_x^4 L_u^{\infty}(\mathbb{R}^2 \times \mathbb{R})}. $$
Then, we just need to prove
\begin{align}\label{eq:3.333}
\left\|F_{j,k} f\right\|_{L_x^4 L_u^{\infty}(\mathbb{R}^2 \times \mathbb{R})} \lesssim 2^{\left(-\frac{1}{4}+\epsilon\right)k} \|f\|_{L_x^4(\mathbb{R}^2)}
\end{align}
for any $\epsilon \in (0, \infty)$. It is easy to see that the desired decay bound at $(\frac{1}{4}, \frac{1}{4})$ can be obtained by letting $\epsilon\in (0,\frac{1}{4})$ in \eqref{eq:3.333}.

To obtain \eqref{eq:3.333}, we first establish the space-time estimates for \emph{Fourier integral operators}, defined on the space of Schwartz functions \(\mathcal{S}(\mathbb{R}^n)\), of the form
\begin{align}\label{eq:1.a}
\mathcal{F}f(z) := \int_{\mathbb{R}^n} e^{i\phi(z, \xi)} a(z, \xi) \hat{f}(\xi) \, \mathrm{d}\xi,
\end{align}
where \(a(z, \xi)\) is a symbol of order \(\sigma \in \mathbb{R}\). Here and hereafter, \(z\) denotes a vector in \(\mathbb{R}^n \times \mathbb{R}\) comprising the space-time variables \((x, u)\). Assume that \(\textrm{supp}~a(\cdot, \xi)\) is contained in a fixed compact set and \(\phi(z, \cdot)\) is homogeneous of degree 1. We say that \(\mathcal{F}\) satisfies the \emph{cinematic curvature condition} if the following hold:
\begin{enumerate}
\item[$\blacksquare$] \textbf{Non-degeneracy condition}: For all \((z, \xi) \in \textrm{supp}~a\),
\begin{align*}
\textrm{rank}~\partial^2_{z\xi}\phi(z, \xi) = n.
\end{align*}

\item[$\blacksquare$] \textbf{Cone condition}: Consider the Gauss map \(G: \textrm{supp}~a \rightarrow \mathbb{S}^n\) defined by
\begin{align*}
G(z, \xi) := \frac{G_0(z, \xi)}{|G_0(z, \xi)|}, \quad \text{where} \quad G_0(z, \xi) := \bigwedge_{j=1}^n \partial_{\xi_j} \partial_z \phi(z, \xi).
\end{align*}
The curvature condition
\begin{align*}
\textrm{rank}~\partial^2_{\xi\xi} \langle \partial_z \phi(z, \xi), G(z, \xi_0) \rangle \big|_{\xi = \xi_0} = n - 1
\end{align*}
holds for all \((z, \xi_0) \in \textrm{supp}~a\).
\end{enumerate}

We now establish the following two lemmas, which will be used repeatedly in the remaining parts of this paper.

\begin{lemma}\label{lemma5.1}
Let \(1 \leq p \leq q \leq \infty\) and \(s \in \mathbb{R}\). Let $P_k$ be the Littlewood-Paley operator defined in Section $1$, and \(\mathcal{F}\) be the Fourier integral operator defined in \eqref{eq:1.a} with a symbol of order \(\sigma\) satisfying the cinematic curvature condition, together with the special assumption that \(\textrm{rank}~\partial^2_{\xi\xi} \phi(z, \xi) = n - 1\) for \(\xi \neq 0\). Then, there exists a positive constant \(C_{(u)}\) such that
\begin{align*}
\left\|\mathcal{F} P_k f\right\|_{L_x^q(\mathbb{R}^n)} \leq C_{(u)} 2^{sk} \|f\|_{L_x^p(\mathbb{R}^n)},
\end{align*}
where \(s \geq \bar{s}_{n, \sigma}(p, q)\) and \(\bar{s}_{n, \sigma}(p, q)\) is defined as
\begin{align*}
\bar{s}_{n, \sigma}(p, q) :=
\begin{cases}
\frac{1}{p} - \frac{n}{q} + \frac{n - 1}{2} + \sigma, & \text{for } q > p'; \\
\frac{n}{p} - \frac{1}{q} - \frac{n - 1}{2} + \sigma, & \text{for } q \leq p'.
\end{cases}
\end{align*}
\end{lemma}

\begin{proof}
Indeed, this lemma can be followed from \cite[Proposition 3.2]{LLLY} by interpolation. Here, we give another proof, which can be also applied to some other estimates for \(\mathcal{F}\).

The case $q=p\in (1,\infty)$ has been proved in Seeger, Sogge and Stein \cite[Theorem 2.1]{SSS}. Then, by interpolation, it suffices to prove
\begin{align}\label{eq:3.da}
\left\|\mathcal{F}P_kf\right\|_{L_{x}^{\infty}(\mathbb{R}^n)}\leq C_{(u)} 2^{(\frac{n+1}{2}+\sigma)k}\left\|f\right\|_{L_x^1(\mathbb{R}^{n})},
\end{align}
\begin{align}\label{eq:3.ca}
\left\|\mathcal{F}P_kf\right\|_{L_{x}^{\infty}(\mathbb{R}^n)}\leq C_{(u)}2^{(\frac{n-1}{2}+\sigma)k}\left\|f\right\|_{L_x^{\infty}(\mathbb{R}^{n})},
\end{align}
and
\begin{align}\label{eq:3.aa}
\left\|\mathcal{F}P_kf\right\|_{L_{x}^{1}(\mathbb{R}^n)}\leq C_{(u)}2^{(\frac{n-1}{2}+\sigma)k}\left\|f\right\|_{L_x^{1}(\mathbb{R}^{n})}.
\end{align}

To establish \eqref{eq:3.da}, by changing variables, we express the kernel of $\mathcal{F}P_k$ as
\begin{align*}
K_{k}(z,y):=2^{nk}\int_{\mathbb{R}^n} e^{i2^k(\phi(z,\xi)-y\cdot\xi)}a(z,2^k\xi)  \bar{\psi} (|\xi|)\,\textrm{d}\xi.
\end{align*}
Without loss of generality, we may suppose that
$\textrm{rank}~\partial_{\xi'\xi'}^{2}\phi(z,\xi)=n-1$ where $\xi:=(\xi_1,\xi')$ with $\xi':=(\xi_2,\cdots,\xi_n)$ is perpendicular to $\xi_1$.
By partition of unity, we can further suppose that $\bar{\psi}=0$ if $|\frac{\xi}{|\xi|}-e_{1}|\geq \frac{1}{10^2}$ and, for all $z\in \textrm{supp}~a(\cdot,\xi)$ and $y\in \mathbb{R}^n$, there is at most one $v\in \mathbb{S}^{n-1}$ satisfying $|v-e_{1}|\leq \frac{1}{9^2}$ such that
\begin{align}\label{eq:4.3a}
\nabla_{\xi}\phi(z,v)=y,
\end{align}
where $e_{1}:=(1,0,\cdots,0)\in \mathbb{S}^{n-1}$. If such $v$ does not exist, then the phase function in $K_{k}$ has no critical point. Moreover, by homogeneity, it follows that
\begin{align*}
|K_{k}(z,y)|\leq C_{(u,N)}2^{-Nk}
\end{align*}
if $N$ is large enough, which further leads to \eqref{eq:3.da}.

Now, we suppose that there exists such $v$. Without loss of generality we may assume that $v=e_{1}$, otherwise we can choose new axes so that $v=e_{1}$. Since $\phi(z,\xi)$ is homogeneous of degree $1$ in $\xi$,
\begin{align*}
\nabla_{\xi}\phi(z,\xi)\cdot \xi=\phi(z,\xi);
\end{align*}
hence $\nabla _{\xi}(\partial_{\xi_i}\phi)(z,\xi)\cdot \xi=0$ for any $i=1,2,\cdots,n$, furthermore, $\partial^2_{\xi_i \xi_1}\phi(z,e_{1})=0$ holds for all $i=1,2,\cdots,n$. This, combined with he special assumption that $\textrm{rank}~\partial^2_{\xi'\xi'}\phi(z,\xi)=n-1$ when $\xi\neq 0$, leads to
\begin{align}\label{eq:4.4a}
\det \partial_{\xi'\xi'}^{2}\phi(z,e_{1})\neq0.
\end{align}

It is easy to see that, for any $i,j=1,2,\cdots,n$, $\partial^2_{\xi_i\xi_j}\phi(z,\xi)$ is homogeneous of degree $1$ in $\xi$ and $\partial_{\xi_i}\phi(z,\xi)$ is homogeneous of degree $0$ in $\xi$. Let $\phi_{\xi_{1}}(z,\xi'):=\phi(z,\xi)$. Then, for all $|(\xi_{1},0)|\in \textrm{supp}~\bar{\psi}$, by \eqref{eq:4.3a} and \eqref{eq:4.4a}, we can see that
\begin{align*}
\nabla_{\xi'}\phi_{\xi_{1}}(z,0)=\nabla_{\xi'}\phi(z,(\xi_{1},0))=\nabla_{\xi'}\phi(z,e_{1})=y',
\end{align*}
and
\begin{align*}
\det\partial_{\xi'\xi'}^{2}\phi_{\xi_{1}}(z,0)=\det\partial_{\xi'\xi'}^{2}\phi(z,(\xi_{1},0))=|\xi_{1}|^{-(n-1)}\det\partial_{\xi'\xi'}^{2}\phi(z,e_{1})\neq0,
\end{align*}
where $y:=(y_{1},y')$. Therefore, from Stein \cite[Chapter VIII, Proposition 6]{Stein}, there exists a constant $C$ such that the following identity is valid:
\begin{align}\label{eq:4.99a}
K_{k}(z,y)&=2^{nk}\int_{\mathbb{R}}\int_{\mathbb{R}^{n-1}} e^{i2^k(\phi(z,\xi)-y\cdot\xi)}a(z,2^k\xi)  \bar{\psi} (|\xi|)\,\textrm{d}\xi'\,\textrm{d}\xi_1\\
&=C2^{\frac{n+1}{2}k}\int_{\mathbb{R}}e^{i2^k\left(\phi(z,(\xi_1,0))-y_1\xi_1\right)}a\left(z,(2^k\xi_1,0)\right)  \bar{\psi} (|\xi_1|) \,\textrm{d}\xi_1+O\left(2^{\left(\frac{n-1}{2}+\sigma\right)k}\right),\nonumber
\end{align}
which further leads to $|K_{k}(z,y)|\lesssim 2^{\left(\frac{n+1}{2}+\sigma\right)k}$ and \eqref{eq:3.da} follows easily.

We now turn to the proof of \eqref{eq:3.ca}, which is based on some results obtained in Stein \cite[Chapter IX, Section 4]{Stein}. Let \(\mathcal{F}\) be the Fourier integral operator with a symbol of order $-\frac{n-1}{2}$, it is enough to prove
\begin{align*}
\left\|\mathcal{F}P_kf\right\|_{L_{x}^{\infty}(\mathbb{R}^n)}\leq C_{(u)}\left\|f\right\|_{L_x^{\infty}(\mathbb{R}^{n})},
\end{align*}
which can be reduced to establish
\begin{align}\label{eq:5.1a}
 \left|\int_{\mathbb{R}^n} K_{k}(z,y)\,\textrm{d}y\right|\leq C_{(u)},
\end{align}
where the kernel of $\mathcal{F}P_k$ is given by
\begin{align*}
K_{k}(z,y)=\int_{\mathbb{R}^n} e^{i(\phi(z,\xi)-y\cdot\xi)}a(z,\xi)  \bar{\psi}_k (|\xi|)\,\textrm{d}\xi.
\end{align*}

Next, a further dyadic division is needed. Roughly speaking, we will split each dyadic shell $\{ \xi \in \mathbb{R}^n: 2^{k-3}\leq |\xi| \leq 2^{k+3} \}$ into thin truncated cones of aperture roughly $2^{-\frac{k}{2}}$, and each such truncated cone is essentially an elongated rectangle whose major axis has length roughly $2^k$, while all the other sides have length roughly $2^{\frac{k}{2}}$. Then, we need roughly $2^{\frac{n-1}{2}k}$ such truncated cones to cover the shell $\{ \xi \in \mathbb{R}^n: 2^{k-3}\leq |\xi| \leq 2^{k+3} \}$ (see Stein \cite[Chapter IX, Figure 1]{Stein}). More precisely, for each $k \in \mathbb{N}$, we fix a collection $\{\xi_k^{\nu}\}_{\nu}\subset \mathbb{S}^{n-1}$ satisfy:
\begin{enumerate}
  \item[\rm(1)] $|\xi_k^{\nu}-\xi_k^{\nu'}|\geq 2^{-\frac{k}{2}}$, if $\nu\neq \nu'$;
  \item[\rm(2)] If $\xi\in \mathbb{S}^{n-1}$, then there exists a $\xi_k^{\nu}$ so that $|\xi-\xi_k^{\nu}|< 2^{-\frac{k}{2}}$.
\end{enumerate}
Let
\begin{align*}
\Gamma_k^{\nu}:=\left\{ \xi \in \mathbb{R}^n:  \left|\frac{\xi}{|\xi|}-\xi_k^{\nu}\right| \leq 2^{1-\frac{k}{2}} \right\},
\end{align*}
which is a cone with central direction $\xi_k^{\nu}$. We can then construct an associated partition of unity
\begin{align*}
\sum_{\nu} \chi_k^{\nu}(\xi)=1
\end{align*}
for all $\xi\neq 0$ and $k \in \mathbb{N}$, where $\chi_k^{\nu}$ is homogeneous of degree $0$ in $\xi$, supported in $\Gamma_k^{\nu}$ and satisfy $|\partial_{\xi}^{\alpha} \chi_k^{\nu}(\xi)|\leq C_{(\alpha)} 2^{\frac{|\alpha|}{2}k}|\xi|^{-|\alpha|}$ for all \( \alpha \in \N_0^n  \).

By the definition of $\chi_k^{\nu}$, let
\begin{align*}
K^{\nu}_{k}(z,y):=\int_{\mathbb{R}^n} e^{i(\phi(z,\xi)-y\cdot\xi)}a(z,\xi)  \bar{\psi}_k (|\xi|)\chi_k^{\nu}(\xi)\,\textrm{d}\xi.
\end{align*}
Consequently,
\begin{align}\label{eq:5.2a}
K_{k}=\sum_{\nu}K^{\nu}_{k}.
\end{align}
Furthermore, we choose axes in the $\xi$-space so that $\xi_1$ is in the direction of $\xi_k^{\nu}$ and write
\begin{align*}
K^{\nu}_{k}(z,y)=\int_{\mathbb{R}^n} e^{i(\nabla_{\xi}\phi(z,\xi_k^{\nu})-y)\cdot\xi}b_k^{\nu}(z,\xi) \,\textrm{d}\xi,
\end{align*}
where
\begin{align*}
b_k^{\nu}(z,\xi):=e^{i(\phi(z,\xi)-\nabla_{\xi}\phi(z,\xi_k^{\nu})\cdot\xi)}a(z,\xi)  \bar{\psi}_k (|\xi|)\chi_k^{\nu}(\xi).
\end{align*}

We define the following operator
\begin{align*}
L:=I-2^{2k}\partial^2_{\xi_1\xi_1}-2^k\nabla_{\xi'}.
\end{align*}
As in Stein \cite[Chapter IX, Section 4]{Stein}, for any \( N \in \N_0 \), we have
\begin{align*}
L^N\left(e^{i(\nabla_{\xi}\phi(z,\xi_k^{\nu})-y)\cdot\xi}\right)=\left[1+2^{2k}\left|\left(\nabla_{\xi}\phi(z,\xi_k^{\nu})-y\right)_1\right|^2 +2^k \left|\left(\nabla_{\xi}\phi(z,\xi_k^{\nu})-y\right)'\right|^2 \right]^N e^{i(\nabla_{\xi}\phi(z,\xi_k^{\nu})-y)\cdot\xi},
\end{align*}
where $(\cdot)_1$ denotes the component in the direction $\xi_k^{\nu}$ and $(\cdot)'$ denotes the orthogonal component. Moreover, by the fact that \(\mathcal{F}\) has a symbol of order $-\frac{n-1}{2}$, it follows that
\begin{align*}
L^N\left(b_k^{\nu}(z,\xi)\right)\leq C_{(u,N)} 2^{-\frac{n-1}{2}k}.
\end{align*}
Notice that the support of $b_k^{\nu}(z,\cdot)$ has volume roughly $2^{(1+\frac{n-1}{2})k}$, pass the operator $L$ onto $K^{\nu}_{k}$, we deduce that
\begin{align}\label{eq:5.3a}
\left|K^{\nu}_{k}(z,y)\right|\leq C_{(u,N)}2^k \left[1+2^{k}\left|\left(\nabla_{\xi}\phi(z,\xi_k^{\nu})-y\right)_1\right| +2^{\frac{k}{2}} \left|\left(\nabla_{\xi}\phi(z,\xi_k^{\nu})-y\right)'\right| \right]^{-2N}.
\end{align}
If we take $N$ large enough, it is easy to see that
\begin{align*}
 \left|\int_{\mathbb{R}^n} K^{\nu}_{k}(z,y)\,\textrm{d}y\right|\leq C_{(u)}2^{-\frac{n-1}{2}k}.
\end{align*}
This, combined with \eqref{eq:5.2a} and the fact that there are essentially $2^{\frac{n-1}{2}k}$ terms involved, leads to \eqref{eq:5.1a}, as desired.

The proof of \eqref{eq:3.aa} is very similar to that of \eqref{eq:3.ca}. Indeed, the change of variables $x\rightarrow \nabla_{\xi}\phi(z,\xi_k^{\nu})$ in \eqref{eq:5.3a}, whose Jacobian is bounded from below, leads to
\begin{align*}
 \left|\int_{\mathbb{R}^n} K^{\nu}_{k}(z,y)\,\textrm{d}x\right|\leq C_{(u)}2^{-\frac{n-1}{2}k}.
\end{align*}
Then,
\begin{align*}
 \left|\int_{\mathbb{R}^n} K_{k}(z,y)\,\textrm{d}x\right|\leq C_{(u)},
\end{align*}
which implies \eqref{eq:3.aa}. The proof is thus complete.
\end{proof}

\begin{lemma}\label{lemma3.3}
Let \(n = 2\), \(1 \leq p \leq q \leq \infty\), and \(s \in \mathbb{R}\). Let $P_k$ be the Littlewood-Paley operator defined in Section $1$, and \(\mathcal{F}\) be a Fourier integral operator as in \eqref{eq:1.a} with a symbol of order \(\sigma\) satisfying the cinematic curvature condition, and assume \(\textrm{rank}~\partial^2_{\xi\xi} \phi(z, \xi) = 1\) for \(\xi \neq 0\). Then, for \(P_k\) defined as in Section \ref{section:1} with \(k \in \mathbb{N}\), there exists a constant \(C > 0\) such that
\begin{align*}
\left\|\mathcal{F} P_k f\right\|_{L_x^q L^q_u(\mathbb{R}^2 \times [1,2])} \leq C 2^{sk} \|f\|_{L_x^p(\mathbb{R}^2)},
\end{align*}
where \(s \geq s_{\sigma, \epsilon}(p, q, q)\) and, for any \(\epsilon > 0\),
\begin{align*}
s_{\sigma, \epsilon}(p, q, q) :=
\begin{cases}
\frac{1}{p} - \frac{3}{q} + \frac{1}{2} + \sigma + \epsilon, & \text{for } q \geq 3p'; \\
\frac{3}{2p} - \frac{3}{2q} + \sigma + \epsilon, & \text{for } p' < q < 3p'; \\
\frac{2}{p} - \frac{1}{q} - \frac{1}{2} + \sigma, & \text{for } q \leq p'.
\end{cases}
\end{align*}
\end{lemma}

\begin{proof}
For any \(\epsilon > 0\), by interpolation, the key estimates we shall make are:
\begin{enumerate}
  \item[\rm(a)] $\left\|\mathcal{F} P_k f\right\|_{L_x^{\infty} L^{\infty}_u(\mathbb{R}^2 \times [1,2])} \lesssim 2^{\left(\frac{3}{2} + \sigma + \epsilon\right)k} \|f\|_{L_x^1(\mathbb{R}^2)}$;
  \item[\rm(b)] $\left\|\mathcal{F} P_k f\right\|_{L_x^{\infty} L^{\infty}_u(\mathbb{R}^2 \times [1,2])} \lesssim 2^{\left(\frac{1}{2} + \sigma + \epsilon\right)k} \|f\|_{L_x^{\infty}(\mathbb{R}^2)}$;
  \item[\rm(c)] $ \left\|\mathcal{F} P_k f\right\|_{L_x^1 L^1_u(\mathbb{R}^2 \times [1,2])} \lesssim 2^{ \left(\frac{1}{2} + \sigma \right)k} \|f\|_{L_x^1(\mathbb{R}^2)}$;
  \item[\rm(d)] $\left\|\mathcal{F} P_k f\right\|_{L_x^2 L^2_u(\mathbb{R}^2 \times [1,2])} \lesssim 2^{\sigma k} \|f\|_{L_x^2(\mathbb{R}^2)}$;
  \item[\rm(e)] $\left\|\mathcal{F} P_k f\right\|_{L_x^4 L^4_u(\mathbb{R}^2 \times [1,2])} \lesssim 2^{(\sigma+\epsilon)k} \|f\|_{L_x^4(\mathbb{R}^2)}$.
\end{enumerate}

The proofs of estimates $(\textrm{a})-(\textrm{d})$ follow directly from Lemma \ref{lemma5.1}, since $C_{(u)}\lesssim 1$ for any $u\in[1,2]$. On the other hand, we want to show that Lee \cite[Corollary 1.5]{Lee06} also implies the estimates $(\textrm{a})-(\textrm{b})$, where the following additional condition is needed:
\begin{center}
all nonzero eigenvalues of $\partial^2_{\xi\xi} \langle \partial_z \phi(z, \xi), G(z, \xi_0) \rangle \big|_{\xi = \xi_0} $ have the same sign.
\end{center}
However, this additional condition is clearly satisfied, since the cone condition $$\textrm{rank}~\partial^2_{\xi\xi} \langle \partial_z \phi(z, \xi), G(z, \xi_0) \rangle \big|_{\xi = \xi_0} =  1$$ holds for all \((z, \xi_0) \in \textrm{supp}~a\) when $n=2$, which further implies that there is only one nonzero eigenvalue of $\partial^2_{\xi\xi} \langle \partial_z \phi(z, \xi), G(z, \xi_0) \rangle \big|_{\xi = \xi_0} $, and so has the same sign.

The estimate $(\textrm{c})$ can also be proved in the similar way as \eqref{eq:3.da}. Indeed, by the fact that $\phi(z,\xi)$ is homogeneous of degree $1$ in $\xi$, we can write \eqref{eq:4.99a} as
\begin{align*}
K_{k}(z,y)=C2^{\frac{3}{2}k}\int_{\mathbb{R}}e^{i2^k\left(\phi(z,e_{1})-y_1\right)\xi_1}a\left(z,(2^k\xi_1,0)\right)  \bar{\psi} (|\xi_1|) \,\textrm{d}\xi_1+O\left(2^{\left(\frac{1}{2}+\sigma\right)k}\right),
\end{align*}
when $n=2$. Then, interpolation by parts shows that
\begin{align*}
|K_{k}(z,y)|\lesssim 2^{\left(\frac{3}{2}+\sigma\right)k}(1+2^{k}|\phi(z,e_{1})-y_{1}|)^{-N}+2^{\left(\frac{1}{2}+\sigma\right)k},
\end{align*}
where $N\in\mathbb{N}$ is large enough. Furthermore, from Mockenhaupt, Seeger and Sogge \cite[P. 70]{MSS}, by the above non-degeneracy condition, we can choose a locally coordinates $z=(x,u)\in \mathbb{R}^2\times \mathbb{R}$ such that
\begin{align*}
 \textrm{rank}~\partial^2_{x\xi}\phi(z,\xi)=2,
\end{align*}
and, if $\xi\neq 0$,
\begin{align*}
\partial_{u}\phi(z,\xi)\neq 0.
\end{align*}
It follows that $|\partial_{u}\phi(z,e_{1})|\gtrsim 1$ for all $z\in \textrm{supp}~a(\cdot,\xi)$. Therefore,
via a change of variables $u\rightarrow \phi(z,e_{1})$, for any $z\in \textrm{supp}~a(\cdot,\xi)$, we assert that
\begin{align*}
\int_{\mathbb{R}}|K_{k}(z,y)|\,\textrm{d}u&\lesssim 2^{\left(\frac{3}{2}+\sigma\right)k}\int_{\mathbb{R}}(1+2^{k}|\phi(z,e_{1})-y_{1}|)^{-N}\,\textrm{d}u+2^{\left(\frac{1}{2}+\sigma\right)k}\\
&\lesssim 2^{\left(\frac{3}{2}+\sigma\right)k}\int_{\mathbb{R}}(1+2^{k}|u-y_{1}|)^{-N}\,\textrm{d}u+2^{\left(\frac{1}{2}+\sigma\right)k}\lesssim2^{\left(\frac{1}{2}+\sigma\right)k}.
\end{align*}
This, combined with the fact that \(\textrm{supp}~a(\cdot, \xi)\) is contained in a fixed compact set, leads to the estimate $(\textrm{c})$.

The estimate $(\textrm{d})$ is a basic result about the Fourier integral operators $\mathcal{F}$, which can be proved by considering the pseudo-differential operator $\mathcal{F}^{\ast}\mathcal{F}$, we refer the reader to $\textrm{H}\ddot{\textrm{o}}\textrm{rmander}$ \cite{Hor}, Stein \cite[Chapter IX, Section 3]{Stein} or Sogge \cite[Section 6.2]{So} for more detailed discussions. The key estimate $(\textrm{e})$ follows from Gao, Liu, Miao and Xi \cite[Theorem 1.4]{GLMX}, which is a generalization of the local smoothing conjecture in $2+1$ dimensions obtained in Guth, Wang and Zhang \cite[Theorem 1.2]{GWZ}. This finishes the proof.
\end{proof}

\begin{remark}
From \cite{LLLY}, we can see that the result in Lemma \ref{lemma3.3} is nearly sharp. Moreover, if we restrict our attention to the Littlewood-Paley operator $P_k$, Lemma \ref{lemma3.3} extends the \(L_x^2(\mathbb{R}^2) \rightarrow L_x^q L^q_u(\mathbb{R}^2 \times \mathbb{R})\) estimate obtained in Mockenhaupt, Seeger and Sogge \cite[Theorem 3.1]{MSS} or the \(L_x^p(\mathbb{R}^2) \rightarrow L_x^q L^q_u(\mathbb{R}^2 \times \mathbb{R})\) estimate with $q \geq 3p'$ and $q\geq \frac{5}{3}p$ proved in Lee \cite[Corollary 1.5]{Lee06} to an almost sharp \(L_x^p(\mathbb{R}^2) \rightarrow L_x^q L^q_u(\mathbb{R}^2 \times \mathbb{R})\) estimate. Specifically, we have
\begin{align*}
\left\|\mathcal{F} P_k f\right\|_{L_x^q L^q_u(\mathbb{R}^2 \times \mathbb{R})} \lesssim \|f\|_{L_x^p(\mathbb{R}^2)}
\end{align*}
with
\begin{align*}
\begin{cases}
\sigma < -\frac{1}{p} + \frac{3}{q} - \frac{1}{2}, & \text{for } q \geq 3p'; \\
\sigma < -\frac{3}{2p} + \frac{3}{2q}, & \text{for } p' < q < 3p'; \\
\sigma \leq -\frac{2}{p} + \frac{1}{q} + \frac{1}{2}, & \text{for } q \leq p'.
\end{cases}
\end{align*}
\end{remark}

\begin{remark}
In \cite{LLLY}, for \(F_{j,k}\) defined in \eqref{eq:2.10}, we established
\begin{align}\label{eq:3.09}
\left\|F_{j,k} f\right\|_{L_x^6 L^6_u(\mathbb{R}^2 \times \mathbb{R})} \lesssim 2^{\epsilon k} \|f\|_{L_x^2(\mathbb{R}^2)}
\end{align}
for any \(\epsilon \in (0, \infty)\). This estimate can be derived by interpolating between \eqref{eq:3.2} and the \(F_{j,k}: L_x^4(\mathbb{R}^2) \rightarrow L_x^4 L^4_u(\mathbb{R}^2 \times \mathbb{R})\) estimate, the latter follows from Gao, Liu, Miao and Xi \cite[Theorem 1.4]{GLMX}. Indeed, \eqref{eq:3.09} can be improved to
\begin{align}\label{eq:3.010}
\left\|F_{j,k} f\right\|_{L_x^6 L^6_u(\mathbb{R}^2 \times \mathbb{R})} \lesssim \|f\|_{L_x^2(\mathbb{R}^2)},
\end{align}
which follows from Seeger, Sogge and Stein \cite[Theorem 3.1]{SSS} by the fact that \(F_{j,k}\) is a Fourier integral operator of order \(-\frac{1}{2}\) satisfying the cinematic curvature condition. We remark that the non-degeneracy assumptions (2.2)--(2.3) and the cone condition (2.6) in Mockenhaupt, Seeger and Sogge \cite[Theorem 3.1]{MSS} are equivalent to the cinematic curvature condition.

On the other hand, we may see from \cite{LLLY} that
the estimate \eqref{eq:3.010} is an endpoint estimate, which implies the sharpness of the $L^2_{\textrm{comp}}(Y)\rightarrow L^q_{\textrm{loc}}(Z)$ local smoothing theorem \cite[Theorem 3.1]{MSS}, and further implies the sharpness of the square function estimate \cite[Theorem 3.2]{MSS} by using Sobolev's embedding theorem in the time variable $u$, at least in two-dimensional situation. Here, we also note that the results in \cite[Theorems 3.1 and 3.2]{MSS} actually are sharp for any $n$-dimensional setting by using the counterexample in Sogge \cite[Section 6.2, P. 185]{So}.
\end{remark}

Now, we turn to the proof of \eqref{eq:3.333}. Notice that $F_{j,k}$ is a Fourier integral operator of order \( -\frac{1}{2} \) and satisfies \eqref{eq:2.100}, then the estimate $(\textrm{e})$ in Lemma \ref{lemma3.3} implies
\begin{align}\label{eq:30.001}
\left\|F_{j,k} f\right\|_{L_x^4 L^4_u(\mathbb{R}^2 \times [1,2])} \lesssim 2^{(-\frac{1}{2}+\epsilon)k} \|f\|_{L_x^4(\mathbb{R}^2)}
\end{align}
for any $\epsilon \in (0, \infty)$. Similarly, we can obtain that
\begin{align}\label{eq:30.002}
\left\|\partial_u (F_{j,k} f)\right\|_{L_x^4 L^4_u(\mathbb{R}^2 \times [1,2])} \lesssim 2^{(\frac{1}{2}+\epsilon)k} \|f\|_{L_x^4(\mathbb{R}^2)}.
\end{align}
Therefore, as an immediate consequence of \eqref{eq:30.001} and \eqref{eq:30.002}, by Lemma \ref{lemma3.2}, we deduce that
\begin{align*}
\left\|F_{j,k} f\right\|_{L_x^4 L^{\infty}_u(\mathbb{R}^2 \times [1,2])} \lesssim \left\|F_{j,k} f\right\|^{\frac{3}{4}}_{L_x^4 L^4_u(\mathbb{R}^2 \times [1,2])}\left\|\partial_u (F_{j,k} f)\right\|^{\frac{1}{4}}_{L_x^4 L^4_u(\mathbb{R}^2 \times [1,2])}\lesssim 2^{(-\frac{1}{4}+\epsilon)k} \|f\|_{L_x^4(\mathbb{R}^2)}
\end{align*}
for any $\epsilon \in (0, \infty)$, which is \eqref{eq:3.333} as desired.

Putting things together, we complete the proof of the sufficiency of the region $(\frac{1}{p}, \frac{1}{q})$ in Theorem \ref{thm1} for all $r \in [1, \infty]$.

\subsection{The necessity of the region of $(\frac{1}{p}, \frac{1}{q})$  for all $r\in[1,\infty]$}\label{subsect:3.2}

In this subsection, we establish the necessary conditions for Theorem \ref{thm1}. The condition \(0 \leq \frac{1}{q}\) is trivial, and the condition \(\frac{1}{q} \leq \frac{1}{p}\) is a standard necessity following from a theorem of H\"ormander \cite{Ho}.

We now demonstrate that, if
\begin{align}\label{eq:3.8}
\left\|Tf\right\|_{L_x^q L^r_u(\mathbb{R}^2 \times [1,2])} \lesssim \|f\|_{L_x^p(\mathbb{R}^2)}
\end{align}
holds, then \((\frac{1}{p}, \frac{1}{q})\) satisfies
\begin{align*}
1 + (1 + \omega)\left(\frac{1}{q} - \frac{1}{p}\right) \geq 0.
\end{align*}

The proof is similar to \cite{LLLY}. Specifically, it suffices to show
\begin{align}\label{eq:3.9}
t_i \left(t_i \left|\gamma(t_i)\right|\right)^{\frac{1}{q} - \frac{1}{p}} \lesssim 1,
\end{align}
where the implicit constant is independent of \(i\), and \(\{t_i\}_{i=1}^{\infty}\)  strictly decreasing and converging to \(0\). Indeed, let \(S_{t_i} := [-t_i, t_i] \times [-|\gamma(t_i)|, |\gamma(t_i)|]\). As shown in \cite{LLLY} that
\begin{align*}
\delta t_i \chi_{\frac{S_{t_i}}{2}}(x) \leq T \chi_{S_{t_i}}(x, u)
\end{align*}
holds for some \(\delta\) sufficiently small. Combining this with H\"older's inequality and \eqref{eq:3.8}, we obtain
\begin{align*}
\int_1^2 \left|\frac{S_{t_i}}{2}\right| \, \mathrm{d}u \leq \frac{1}{\delta t_i} \int_{\frac{S_{t_i}}{2}} \left\|T \chi_{S_{t_i}}(x, u)\right\|_{L^1_u([1,2])} \, \mathrm{d}x
\lesssim \frac{|S_{t_i}|^{\frac{1}{q'}}}{t_i} \left\|T \chi_{S_{t_i}}\right\|_{L_x^q L^1_u(\mathbb{R}^2 \times [1,2])}
\lesssim \frac{|S_{t_i}|^{\frac{1}{q'}}}{t_i} \left\|\chi_{S_{t_i}}\right\|_{L_x^p(\mathbb{R}^2)}.
\end{align*}
Since \(|S_{t_i}| \approx t_i |\gamma(t_i)|\) and \(\|\chi_{S_{t_i}}\|_{L_x^p(\mathbb{R}^2)} \approx (t_i |\gamma(t_i)|)^{\frac{1}{p}}\), the desired inequality \eqref{eq:3.9} follows.

To obtain the other necessary conditions for Theorem \ref{thm1}, we also require the following proposition, which is based on Lemmas \ref{lemma3.2}, \ref{lemma5.1} and \ref{lemma3.3}.

\begin{proposition}\label{proposition3.4}
Let \(n = 2\), \(1 \leq p \leq q \leq \infty\), and \(\mathcal{F}\) be as in Lemma \ref{lemma3.3}. Then, there exists a constant \(C > 0\) such that
\begin{align*}
\left\|\mathcal{F} P_k f\right\|_{L_x^q L^\infty_u(\mathbb{R}^2 \times \mathbb{R})} \leq C 2^{sk} \|f\|_{L_x^p(\mathbb{R}^2)},
\end{align*}
where \(s \geq s_{\sigma, \epsilon}(p, q, \infty)\) and, for any \(\epsilon > 0\),
\begin{align*}
s_{\sigma, \epsilon}(p, q, \infty) :=
\begin{cases}
\frac{1}{p} - \frac{2}{q} + \frac{1}{2} + \sigma + \epsilon, & \text{for } q \geq 3p'; \\
\frac{3}{2p} - \frac{1}{2q} + \sigma + \epsilon, & \text{for } p' < q < 3p'; \\
\frac{2}{p} - \frac{1}{2} + \sigma, & \text{for } q \leq p'.
\end{cases}
\end{align*}
\end{proposition}

\begin{proof}
Notice that $\phi(z,\cdot)$ is a homogeneous function of degree $1$, which leads to that $\partial_u\phi(z,\cdot)$ is also a homogeneous function of degree $1$ and $|\partial_u\phi(z,\cdot)|\lesssim |\xi|$. Then, the proof of the $L_x^p(\mathbb{R}^2)\rightarrow L_x^qL^q_u(\mathbb{R}^2\times \mathbb{R})$ estimate for $\partial_u(\mathcal{F}P_kf)$ is very similar as that for $\mathcal{F}P_kf$, but need to multiply a extra factor $2^k$. Therefore, by Lemma \ref{lemma3.2} with $r=q$ and Lemma \ref{lemma3.3}, we can say
\begin{align*}
s_{\sigma,\epsilon}(p,q,\infty)= s_{\sigma,\epsilon}(p,q,q)+\frac{1}{q}
\end{align*}
 and complete the proof of this proposition.
\end{proof}

Let \(Q_k f := (\varrho_k(\xi) \hat{f}(\xi))^\vee\) with \(\varrho_k(\xi) := \varrho(2^{-k} \xi)\) and \(k \in \mathbb{N}\), where \(\varrho \in \mathcal{S}(\mathbb{R}^2)\) and
\begin{align*}
\textrm{supp}~\varrho \subset \left\{\xi \in \mathbb{R}^2: \frac{1}{2} \leq |\xi| \leq 1 \text{ and } \frac{|\xi_1|}{|\xi_2|} \in \textrm{supp}~\rho\right\}.
\end{align*}
Assume that \eqref{eq:3.8} holds for some \(1 \leq p \leq q \leq \infty\) and \(1 \leq r \leq \infty\), and apply it to \(Q_k f\). From the preliminaries in Section 1, we conclude that
\begin{align*}
\left\|F_{j,k} P_k Q_k f\right\|_{L_x^q L^r_u(\mathbb{R}^2 \times \mathbb{R})} \lesssim \|f\|_{L_x^p(\mathbb{R}^2)}.
\end{align*}
Here, we used the fact that \(F_{j,k} P_k Q_k f = F_{j,k} Q_k f\). Furthermore, by Lemma \ref{lemma3.2}, we have
\begin{align}\label{eq:3.10}
\left\|F_{j,k} P_k Q_k f\right\|_{L_x^q L^\infty_u(\mathbb{R}^2 \times \mathbb{R})} \lesssim 2^{\frac{1}{r}k} \|f\|_{L_x^p(\mathbb{R}^2)}.
\end{align}

On the other hand, since \(F_{j,k}\) is a Fourier integral operator of order \(-\frac{1}{2}\) satisfying the special assumption \eqref{eq:2.100}, Proposition \ref{proposition3.4} implies
\begin{align}\label{eq:3.11}
\left\|F_{j,k} P_k Q_k f\right\|_{L_x^q L^\infty_u(\mathbb{R}^2 \times \mathbb{R})} \lesssim 2^{sk} \|f\|_{L_x^p(\mathbb{R}^2)}
\end{align}
with \(s \geq s_{-\frac{1}{2}, \epsilon}(p, q, \infty)\), where

\begin{align*}
s_{-\frac{1}{2}, \epsilon}(p, q, \infty) =
\begin{cases}
\frac{1}{p} - \frac{2}{q} + \epsilon, & \text{for } q \geq 3p'; \\
\frac{3}{2p} - \frac{1}{2q} - \frac{1}{2} + \epsilon, & \text{for } p' < q < 3p'; \\
\frac{2}{p} - 1, & \text{for } q \leq p'.
\end{cases}
\end{align*}
Here \(\epsilon \in (0, \infty)\) is an arbitrary constant. Moreover, from \cite{LiuYu}, the estimates \eqref{eq:3.1}, \eqref{eq:3.5}, and \eqref{eq:3.6} represent the endpoint estimates for \(F_{j,k}\) when taking the \(L^\infty\)-norm in \(u\). i.e., the points \((0, 0)\), \((\frac{1}{2}, \frac{1}{2})\), and \((\frac{2}{5}, \frac{1}{5})\) are the endpoints. Consequently, \(s_{-\frac{1}{2}, 0}(p, q, \infty)\) is the infimum such that \eqref{eq:3.11} holds. Combining this with \eqref{eq:3.10}, we obtain
$\frac{1}{r} \geq \max \{ \frac{1}{p} - \frac{2}{q}, \frac{3}{2p} - \frac{1}{2q} - \frac{1}{2}, \frac{2}{p} - 1 \}$.
This further implies
\begin{align*}
\max \left\{ \frac{1}{2p} - \frac{1}{2r}, \frac{3}{p} - \frac{r + 2}{r} \right\} \leq \frac{1}{q}
~~
\textrm{and}
~~
\frac{1}{p} \leq \frac{r + 1}{2r},
\end{align*}
as desired.

Therefore, we conclude the proof of the necessity of $(\frac{1}{p}, \frac{1}{q})$ in Theorem \ref{thm1} for all $r\in[1,\infty]$.

\section{Proof of Theorem \ref{thm3}}

In this section, we investigate the time-space estimate, i.e., the \(L_x^p(\mathbb{R}^2) \rightarrow L^r_u L_x^q(\mathbb{R}^2 \times [1,2])\) estimate, for the operator \(T\) defined in \eqref{eq:2.1}. We address the sufficient and necessary parts separately, although the result is not sharp due to certain technical limitations.

\subsection{Sufficient Part}\label{subsect:4.1}

The proof of the time-space estimate for \(T\) differs significantly from the space-time estimate established in Theorem \ref{thm1}. From the preliminaries in Section \ref{preliminaries}, it suffices to prove that
\begin{align*}
\sum_{j \in \mathbb{Z}_0^{-}} 2^j |2^j \gamma(2^j)|^{\frac{1}{q} - \frac{1}{p}} \left\| \widetilde{T}_j \right\|_{L_x^p(\mathbb{R}^2) \rightarrow L^r_u L_x^q(\mathbb{R}^2 \times [1,2])} \lesssim 1.
\end{align*}
By Lemma \ref{lemma 2.3}, it remains to prove the time-space estimate for \(\widetilde{T}_j\) uniformly in \(j \in \mathbb{Z}_0^{-}\). As in Section \ref{preliminaries}, we reduce this estimate to \(\widetilde{T}^1_{j,k}\) defined in \eqref{eq:2.5}. Furthermore, from \cite{LiuYu}, we know that the time-space estimate for \(\widetilde{T}^{1,a}_{j,k}\) is bounded by \(2^{-Nk}\) for \(N\) sufficiently large. We can then reduce to estimate the operator \(\widetilde{T}^{1,b}_{j,k}\), rather than \(F_{j,k}\), in the case \(1 \leq r < p \leq q \leq \infty\).

Indeed, for \(F_{j,k}\) and \(E_{j,k}\) defined in \eqref{eq:2.8} and \eqref{eq:2.08}, respectively, we also decompose \(\mathbb{R}^2\) into \(\bigcup_{i \in \mathbb{Z}^2} B(x_i, \varpi)\) such that \(|x_i - x_{i'}| \approx |i - i'| \varpi\) for any \(i \neq i'\). Unlike \eqref{eq:2.9}, we bound \(\|\widetilde{T}^{1,b}_{j,k} f\|_{L^r_u L_x^q(\mathbb{R}^2 \times [1,2])}\) by
\begin{align}\label{eq:5.1}
\left\|\left[\sum_{i \in \mathbb{Z}^2} \left\|\left(E_{j,k} - F_{j,k}\right)\left(f(\cdot - x_i)\right)\right\|^q_{L_x^q(\mathbb{R}^2)}\right]^{\frac{1}{q}}\right\|_{L_u^r(\mathbb{R})} + \left\|\left[\sum_{i \in \mathbb{Z}^2} \left\|F_{j,k}\left(f(\cdot - x_i)\right)\right\|^q_{L_x^q(\mathbb{R}^2)}\right]^{\frac{1}{q}}\right\|_{L_u^r(\mathbb{R})}.
\end{align}
The estimate for the first term follows from an argument similar to that for \(\widetilde{T}^{1,a}_{j,k}\). The key difference lies in the second term. To be more precise, when \(1 \leq p \leq q \leq r \leq \infty\), Minkowski's inequality implies that the second term can be bounded by
\begin{align*}
\left[\sum_{i \in \mathbb{Z}^2} \left\|F_{j,k}\left(f(\cdot - x_i)\right)\right\|^q_{L_u^r L_x^q(\mathbb{R}^2 \times \mathbb{R})}\right]^{\frac{1}{q}},
\end{align*}
which can be reduced to the corresponding time-space estimate \(F_{j,k}: L_x^p(\mathbb{R}^2) \rightarrow L^r_u L_x^q(\mathbb{R}^2 \times \mathbb{R})\) with a decay in \(k\). When \(1 \leq p \leq r \leq q \leq \infty\), the second term can be bounded as
\begin{align*}
\left\|\left[\sum_{i \in \mathbb{Z}^2} \left\|F_{j,k}\left(f(\cdot - x_i)\right)\right\|^r_{L_x^q(\mathbb{R}^2)}\right]^{\frac{1}{r}}\right\|_{L_u^r(\mathbb{R})} \leq \left[\sum_{i \in \mathbb{Z}^2} \left\|F_{j,k}\left(f(\cdot - x_i)\right)\right\|^p_{L_u^r L_x^q(\mathbb{R}^2 \times \mathbb{R})}\right]^{\frac{1}{p}},
\end{align*}
which also reduces to the corresponding decay time-space estimate for \(F_{j,k}\). Thus, the only case where we do not reduce the estimate to a corresponding decay time-space estimate for \(F_{j,k}\) is when \(1 \leq r < p \leq q \leq \infty\).

We remark that the most challenging aspect of proving the time-space estimates for \(T\) may not lie in the case \(1 \leq r < p \leq q \leq \infty\), since the \(L_x^\infty(\mathbb{R}^2) \rightarrow L^\infty_u L_x^\infty(\mathbb{R}^2 \times [1,2])\) estimate for \(T\) holds trivially, and we can handle the region \(1 \leq r < p \leq q \leq \infty\) via interpolation. The most significant difference is that we can only obtain a common estimate
\begin{align*}
\|F_{j,k} f\|_{L^r_u L_x^\infty(\mathbb{R}^2 \times \mathbb{R})} \lesssim 2^k \|f\|_{L_x^1(\mathbb{R}^2)}
\end{align*}
for all \(r \in [1, \infty]\), unlike the estimate in \eqref{eq:3.3}.

Based on the above discussion, we divide the proof of Theorem \ref{thm3} into three parts.

\subsubsection{The Case \(1 \leq p \leq q \leq r \leq \infty\)}

Observe that \(F_{j,k} P_k f = F_{j,k} f\), and \(F_{j,k}\) is a Fourier integral operator of order \(-\frac{1}{2}\) satisfying \eqref{eq:2.100}. By Lemma \ref{lemma3.3}, we have
\begin{align}\label{eq:5.2}
\left\|F_{j,k} f\right\|_{L^q_u L_x^q(\mathbb{R}^2 \times \mathbb{R})} \lesssim 2^{sk} \|f\|_{L_x^p(\mathbb{R}^2)},
\end{align}
where \(s \geq s_{-\frac{1}{2}, \epsilon}(p, q, q)\) and
\begin{align*}
s_{-\frac{1}{2}, \epsilon}(p, q, q) =
\begin{cases}
\frac{1}{p} - \frac{3}{q} + \epsilon, & \text{for } q \geq 3p'; \\
\frac{3}{2p} - \frac{3}{2q} - \frac{1}{2} + \epsilon, & \text{for } p' < q < 3p'; \\
\frac{2}{p} - \frac{1}{q} - 1, & \text{for } q \leq p'.
\end{cases}
\end{align*}
Additionally, Lemma \ref{lemma5.1} implies\footnote{We remark that: by the following Lemma \ref{lemma5.a}, then the estimate \eqref{eq:5.3} is established with $s\geq s_{-\frac{1}{2}, \epsilon}(p, q, q)+\frac{1}{q}$. However, we have $\bar{s}_{2, -\frac{1}{2}}(p, q)<s_{-\frac{1}{2}, \epsilon}(p, q, q)+\frac{1}{q}$.}
\begin{align}\label{eq:5.3}
\left\|F_{j,k} f\right\|_{L^\infty_u L_x^q(\mathbb{R}^2 \times \mathbb{R})} \lesssim 2^{sk} \|f\|_{L_x^p(\mathbb{R}^2)},
\end{align}
where \(s \geq \bar{s}_{2, -\frac{1}{2}}(p, q)\) with
\begin{align*}
\bar{s}_{2, -\frac{1}{2}}(p, q) :=
\begin{cases}
\frac{1}{p} - \frac{2}{q}, & \text{for } q > p'; \\
\frac{2}{p} - \frac{1}{q} - 1, & \text{for } q \leq p'.
\end{cases}
\end{align*}

Interpolating between \eqref{eq:5.2} and \eqref{eq:5.3}, we obtain
\begin{align}\label{eq:5.4}
\left\|F_{j,k} f\right\|_{L^r_u L_x^q(\mathbb{R}^2 \times \mathbb{R})} \lesssim 2^{sk} \|f\|_{L_x^p(\mathbb{R}^2)}
\end{align}
for all \(1 \leq p \leq q \leq r \leq \infty\), where \(s \geq \tilde{s}_{A, \epsilon}(p, q, r)\) and, for any \(\epsilon > 0\),
\begin{align*}
\tilde{s}_{A, \epsilon}(p, q, r) :=
\begin{cases}
\frac{1}{p} - \frac{2}{q} - \frac{1}{r} + \epsilon, & \text{for } q \geq 3p'; \\
\frac{1}{p} - \frac{2}{q} + \left(\frac{1}{2p} + \frac{1}{2q} - \frac{1}{2}\right) \frac{q}{r} + \epsilon, & \text{for } p' < q < 3p'; \\
\frac{2}{p} - \frac{1}{q} - 1, & \text{for } q \leq p'.
\end{cases}
\end{align*}
Therefore, when \((\frac{1}{p}, \frac{1}{q})\) satisfies \(\tilde{s}_{A, 0}(p, q, r) < 0\), we obtain a desired decay time-space estimate for \(F_{j,k}\), which in turn implies the time-space estimate for \(T\).

\subsubsection{The Case \(1 \leq p \leq r \leq q \leq \infty\)}

One can apply the Sobolev embedding theorem with respect to the space variable $x$ to see that
\begin{align*}
\left\|F_{j,k} f\right\|_{L^r_u L_x^q(\mathbb{R}^2 \times \mathbb{R})} \lesssim 2^{2\left(\frac{1}{r} - \frac{1}{q}\right)k} \left\|F_{j,k} f\right\|_{L^r_u L_x^r(\mathbb{R}^2 \times \mathbb{R})}
\end{align*}
for \(r \leq q \leq \infty\). Combining this with \eqref{eq:5.4}, we derive
\begin{align}\label{eq:5.5}
\left\|F_{j,k} f\right\|_{L^r_u L_x^q(\mathbb{R}^2 \times \mathbb{R})} \lesssim 2^{sk} \|f\|_{L_x^p(\mathbb{R}^2)}
\end{align}
for all \(1 \leq p \leq r \leq q \leq \infty\), where \(s \geq \tilde{s}_{A, \epsilon}(p, r, r) + 2\left(\frac{1}{r} - \frac{1}{q}\right)\).

On the other hand, by H\"older's inequality and \eqref{eq:5.4}, a straightforward computation shows that
\begin{align*}
\left\|F_{j,k} f\right\|_{L^r_u L_x^q(\mathbb{R}^2 \times \mathbb{R})} \lesssim \left\|F_{j,k} f\right\|_{L^q_u L_x^q(\mathbb{R}^2 \times \mathbb{R})} \lesssim 2^{sk} \|f\|_{L_x^p(\mathbb{R}^2)}
\end{align*}
holds for all \(1 \leq p \leq r \leq q \leq \infty\) with \(s \geq \tilde{s}_{A, \epsilon}(p, q, q)\). Combining this with \eqref{eq:5.5}, we obtain
\begin{align}\label{eq:5.05}
\left\|F_{j,k} f\right\|_{L^r_u L_x^q(\mathbb{R}^2 \times \mathbb{R})} \lesssim 2^{sk} \|f\|_{L_x^p(\mathbb{R}^2)}
\end{align}
for all \(1 \leq p \leq r \leq q \leq \infty\), where \(s \geq \tilde{s}_{B, \epsilon}(p, q, r)\) with
\[
\tilde{s}_{B, \epsilon}(p, q, r) := \min\left\{\tilde{s}_{A, \epsilon}(p, r, r) + 2\left(\frac{1}{r} - \frac{1}{q}\right), \tilde{s}_{A, \epsilon}(p, q, q)\right\}.
\]
Then, when \((\frac{1}{p}, \frac{1}{q})\) satisfies \(\tilde{s}_{B, 0}(p, q, r) < 0\), we obtain a desired decay time-space estimate for \(F_{j,k}\), which implies the time-space estimate for \(T\).

\subsubsection{The Case \(1 \leq r < p \leq q \leq \infty\)}

As in \cite{LiuYu}, by \eqref{eq:5.1}, the corresponding time-space estimate for \(E_{j,k} - F_{j,k}\) has a bound of \(2^{-Nk}\) for \(N\) sufficiently large, for all \(1 \leq p \leq q \leq \infty\) and \(1 \leq r \leq \infty\). Thus, the estimate in \eqref{eq:5.5} also applies to \(\widetilde{T}^{1,b}_{j,k}\). Combining this with the time-space estimate for \(\widetilde{T}^{1,a}_{j,k}\), we obtain
\begin{align}\label{eq:5.6}
\left\|\widetilde{T}^1_{j,k} f\right\|_{L^r_u L_x^q(\mathbb{R}^2 \times \mathbb{R})} \lesssim 2^{sk} \|f\|_{L_x^p(\mathbb{R}^2)}
\end{align}
for all \(1 \leq p \leq r \leq q \leq \infty\) with \(s \geq \tilde{s}_{B, \epsilon}(p, q, r)\).

Now, consider the case \(1 \leq r < p \leq q \leq \infty\). By H\"older's inequality and \eqref{eq:5.6}, we have
\begin{align*}
\left\|\widetilde{T}^1_{j,k} f\right\|_{L^r_u L_x^q(\mathbb{R}^2 \times \mathbb{R})} \lesssim \left\|\widetilde{T}^1_{j,k} f\right\|_{L^p_u L_x^q(\mathbb{R}^2 \times \mathbb{R})} \lesssim 2^{\tilde{s}_{B, \epsilon}(p, q, p)k} \|f\|_{L_x^p(\mathbb{R}^2)}
\end{align*}
and
\begin{align*}
\left\|\widetilde{T}^1_{j,k} f\right\|_{L^r_u L_x^q(\mathbb{R}^2 \times \mathbb{R})} \lesssim \left\|\widetilde{T}^1_{j,k} f\right\|_{L^q_u L_x^q(\mathbb{R}^2 \times \mathbb{R})} \lesssim 2^{\tilde{s}_{B, \epsilon}(p, q, q)k} \|f\|_{L_x^p(\mathbb{R}^2)}.
\end{align*}
Thus, for all \(1 \leq r < p \leq q \leq \infty\),
\begin{align}\label{eq:5.7}
\left\|\widetilde{T}^1_{j,k} f\right\|_{L^r_u L_x^q(\mathbb{R}^2 \times \mathbb{R})} \lesssim 2^{sk} \|f\|_{L_x^p(\mathbb{R}^2)},
\end{align}
where \(s \geq \tilde{s}_{C, \epsilon}(p, q, r)\) and
\[
\tilde{s}_{C, \epsilon}(p, q, r) := \min\left\{\tilde{s}_{B, \epsilon}(p, q, p), \tilde{s}_{B, \epsilon}(p, q, q)\right\}.
\]
Notice that $$\tilde{s}_{C, 0}(p, q, r) = \min\left\{\tilde{s}_{A, 0}(p, p, p) + 2\left(\frac{1}{p} - \frac{1}{q}\right), \tilde{s}_{A, 0}(p, q, q)\right\}.$$ Consequently, if \((\frac{1}{p}, \frac{1}{q})\) satisfies \(\tilde{s}_{C, 0}(p, q, r) < 0\), we establish the desired decay time-space estimate for \(\widetilde{T}^1_{j,k}\), and the corresponding time-space estimate for \(T\) follows.

\begin{remark}\label{remark5.1}
For any \(\epsilon > 0\), a computation yields
\begin{align*}
\tilde{s}_{C, \epsilon}(p, q, r) =
\begin{cases}
-\frac{2}{q} + \epsilon, & \text{for } q \geq 3p' \text{ and } \frac{1}{p} \leq \frac{1}{4}; \\
\frac{2}{p} - \frac{2}{q} - \frac{1}{2} + \epsilon, & \text{for } q \geq 3p' \text{ and } \frac{1}{4} < \frac{1}{p} < \frac{1}{2} - \frac{1}{q}; \\
\frac{1}{p} - \frac{3}{q} + \epsilon, & \text{for } q \geq 3p' \text{ and } \frac{1}{p} \geq \frac{1}{2} - \frac{1}{q}; \\
\frac{3}{2p} - \frac{3}{2q} - \frac{1}{2} + \epsilon, & \text{for } p' < q < 3p'; \\
\frac{2}{p} - \frac{1}{q} - 1 + \epsilon, & \text{for } q \leq p'.
\end{cases}
\end{align*}
If \(\tilde{s}_{C, 0}(p, q, r) < 0\), then \eqref{eq:5.7} holds with a decay in \(k\), where \((\frac{1}{p}, \frac{1}{q})\) belongs to the region with vertices \((0, 0)\), \((\frac{1}{4}, 0)\), \((\frac{3}{8}, \frac{1}{8})\), \((\frac{1}{2}, \frac{1}{6})\), \((\frac{2}{3}, \frac{1}{3})\), and \((1, 1)\). Furthermore, the point \((\frac{3}{8}, \frac{1}{8})\) can be removed by interpolation between \((\frac{1}{4}, 0)\) and \((\frac{1}{2}, \frac{1}{6})\). Thus, \eqref{eq:5.7} holds with some \(s < 0\) if \((\frac{1}{p}, \frac{1}{q})\) satisfies
\begin{align*}
\begin{cases}
\frac{1}{q} > \frac{2}{3p} - \frac{1}{6}, & \text{for } q \geq 3p'; \\
\frac{1}{q} > \frac{1}{p} - \frac{1}{3}, & \text{for } p' < q < 3p'; \\
\frac{1}{q} > \frac{2}{p} - 1, & \text{for } q \leq p'.
\end{cases}
\end{align*}
Combining \(1 \leq r < p \leq q \leq \infty\) with the necessary condition \(1 + (1 + \omega)\left(\frac{1}{q} - \frac{1}{p}\right) > 0\), the region of \((\frac{1}{p}, \frac{1}{q})\) in Theorem \ref{thm3} for this case is illustrated in Figure \ref{Figure:5}. Here, we assume \(1 \leq r \leq \frac{3}{2}\), and the figures for other ranges of \(r\) can be constructed similarly.

\begin{figure}[htbp]
  \centering
  \includegraphics[width=4in]{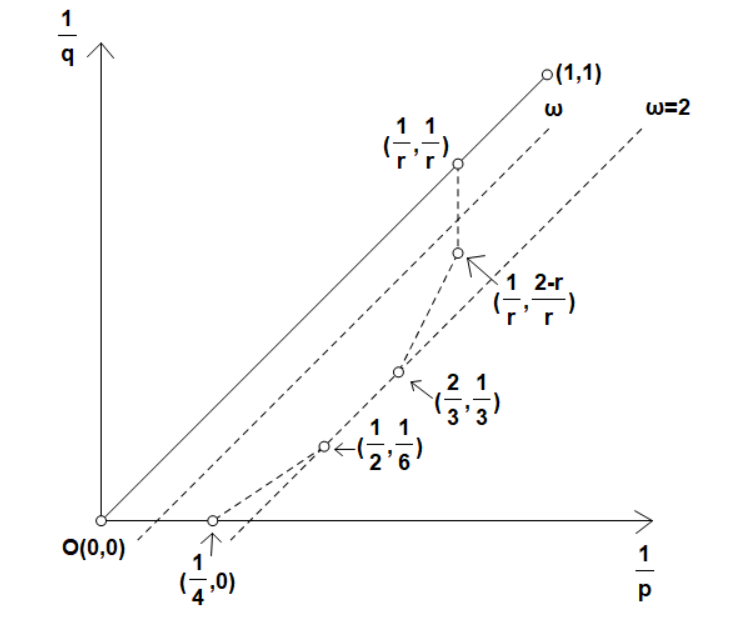}
  \caption{The region of boundedness in \eqref{eq:1.7} for the case \(1 \leq r < p \leq q \leq \infty\) and \(1 \leq r \leq \frac{3}{2}\).}
  \label{Figure:5}
\end{figure}
\end{remark}

\subsection{Necessary Part}\label{subsect:4.2}

In this subsection, we establish necessary conditions. We begin with a lemma analogous to Lemma \ref{lemma3.2}.

\begin{lemma}\label{lemma5.a}
Suppose $G(x,u) \in C_c^1\left(\frac{1}{2}, \frac{5}{2}\right)$ with respect to $u$ for all $x \in \mathbb{R}^2$. Then, for any $1 \leq q, r \leq \infty$, the following inequality holds:
\begin{align*}
\left\|G\right\|_{L^{\infty}_u L_x^{q}(\mathbb{R}^2 \times \mathbb{R})} \lesssim \left\|G\right\|^{\frac{1}{r'}}_{L^{r}_u L_x^{q}(\mathbb{R}^2 \times \mathbb{R})} \left\|\partial_u G\right\|^{\frac{1}{r}}_{L^{r}_u L_x^{q}(\mathbb{R}^2 \times \mathbb{R})}.
\end{align*}
\end{lemma}

\begin{proof}
It suffices to consider $r \in [1, \infty)$. Observe that for all $u \in \left(\frac{1}{2}, \frac{5}{2}\right)$,
\begin{align*}
\left\|G(x,u)\right\|_{L_x^{q}(\mathbb{R}^2)}^r \lesssim \int_\mathbb{R} \left\|G(x,v)\right\|_{L_x^{q}(\mathbb{R}^2)}^{r-1} \left|\partial_v \|G(x,v)\|_{L_x^{q}(\mathbb{R}^2)}\right| \, \mathrm{d}v.
\end{align*}
Applying H\"older's inequality, we obtain
\begin{align*}
\left\|G(x,u)\right\|_{L^{\infty}_u L_x^{q}(\mathbb{R}^2 \times \mathbb{R})}^r \lesssim \left\|G(x,u)\right\|_{L^{r}_u L_x^{q}(\mathbb{R}^2 \times \mathbb{R})}^{r-1} \left\|\partial_u \|G(x,u)\|_{L_x^{q}(\mathbb{R}^2)}\right\|_{L^{r}_u(\mathbb{R})}.
\end{align*}
Combining this with the inequality
\begin{align*}
\left|\partial_u \|G(x,u)\|_{L_x^{q}(\mathbb{R}^2)}\right| \leq \|\partial_u G(x,u)\|_{L_x^{q}(\mathbb{R}^2)}
\end{align*}
for all $u\in(\frac{1}{2}, \frac{5}{2})$, which follows from H\"older's inequality, yields the desired estimate. This completes the proof.
\end{proof}

Now, assume that
\begin{align}\label{eq:5.a}
\left\|Tf\right\|_{L^r_u L_x^q(\mathbb{R}^2 \times [1,2])} \lesssim \|f\|_{L_x^{p}(\mathbb{R}^2)}
\end{align}
holds for some $1 \leq p, q, r \leq \infty$.

\subsubsection{Proof of (I) in Theorem \ref{thm3}}

We restrict our attention to the case $q \in [1, \infty)$. Let $f \in C_c^{\infty}(\mathbb{R}^2)$ with $\|f\|_{L_x^{p}(\mathbb{R}^2)} = 1$. Since $T$ is translation-invariant for any fixed $u \in [1,2]$, it follows from \eqref{eq:5.a} that
\begin{align*}
\left\|Tf(x-h,u) + Tf(x,u)\right\|_{L^r_u L_x^q(\mathbb{R}^2 \times [1,2])} \leq \Theta \|f(x-h) + f(x)\|_{L_x^{p}(\mathbb{R}^2)},
\end{align*}
where $h \in \mathbb{R}^2$ and $\Theta := \sup_{f \in C_c^{\infty}(\mathbb{R}^2), \|f\|_{L_x^{p}(\mathbb{R}^2)} = 1} \|Tf\|_{L^r_u L_x^q(\mathbb{R}^2 \times [1,2])} < \infty$. Letting $|h| \to \infty$ and applying Grafakos \cite[Exercise 2.5.1]{Gra}, we deduce
\begin{align*}
2^{\frac{1}{q}} \left\|Tf(x,u)\right\|_{L^r_u L_x^q(\mathbb{R}^2 \times [1,2])} \leq \liminf_{|h| \to \infty} \left\|Tf(x-h,u) + Tf(x,u)\right\|_{L^r_u L_x^q(\mathbb{R}^2 \times [1,2])}
\leq 2^{\frac{1}{p}} \Theta \|f\|_{L_x^{p}(\mathbb{R}^2)}.
\end{align*}
Thus, $\Theta \leq 2^{\frac{1}{p} - \frac{1}{q}} \Theta$, which is impossible unless $q \geq p$ or $T$ is the zero operator. This contradiction establishes $\frac{1}{q} \leq \frac{1}{p}$, completing the proof of (I) of Theorem \ref{thm3}.

\subsubsection{Proof of (II) of Theorem \ref{thm3}}

As in Subsection \ref{subsect:3.2}, if \eqref{eq:5.a} holds, then
\begin{align*}
\left\|Tf\right\|_{L^1_u L_x^q(\mathbb{R}^2 \times [1,2])} \lesssim \|f\|_{L_x^{p}(\mathbb{R}^2)}.
\end{align*}
Furthermore,
\begin{align*}
\int_1^2 \left|\frac{S_{t_i}}{2}\right| \, \mathrm{d}u \leq \left\|\frac{1}{\delta t_i} \int_{\frac{S_{t_i}}{2}} T\chi_{S_{t_i}}(x,u) \, \mathrm{d}x\right\|_{L^1_u([1,2])} \lesssim \frac{|S_{t_i}|^{\frac{1}{q'}}}{t_i} \left\|T\chi_{S_{t_i}}\right\|_{L^1_u L_x^{q}(\mathbb{R}^2 \times [1,2])} \lesssim \frac{|S_{t_i}|^{\frac{1}{q'}}}{t_i} \left\|\chi_{S_{t_i}}\right\|_{L_x^{p}(\mathbb{R}^2)}.
\end{align*}
Using the facts that $|S_{t_i}| \approx t_i |\gamma(t_i)|$ and $\|\chi_{S_{t_i}}\|_{L_x^{p}(\mathbb{R}^2)} \approx (t_i |\gamma(t_i)|)^{\frac{1}{p}}$, we obtain \eqref{eq:3.9}. This implies $1 + (1 + \omega)\left(\frac{1}{q} - \frac{1}{p}\right) \geq 0$, proving  (II) of Theorem \ref{thm3}.

\subsubsection{Proof of (III) of Theorem \ref{thm3}}

This proof is inspired by examples in \cite{LLLY}. For any $u \in [1,2]$ and $\epsilon > 0$, let $D_u$ be the $\epsilon$-neighborhood of the curve $(ut, u\gamma(t))$ with $t \in (0,1]$. As in \cite{LLLY}, there exists $\delta \in (0, \infty)$ such that $T(\chi_{B(0,\delta \epsilon)})(x,u) \geq \epsilon \chi_{D_u}(x)$. Since $|D_u| \approx \epsilon$ for all $u \in [1,2]$, we have
\begin{align*}
\left\|T(\chi_{B(0,\delta \epsilon)})\right\|_{L_x^q(\mathbb{R}^2)} \geq \epsilon \left\|\chi_{D_u}\right\|_{L_x^{q}(\mathbb{R}^2)} \gtrsim \epsilon^{1 + \frac{1}{q}},
\end{align*}
which implies
\begin{align*}
\left\|T(\chi_{B(0,\delta \epsilon)})\right\|_{L^r_u L_x^q(\mathbb{R}^2 \times [1,2])} \gtrsim \epsilon^{1 + \frac{1}{q}}.
\end{align*}
Combining this with $\|Tf\|_{L^r_u L_x^q(\mathbb{R}^2 \times [1,2])} \lesssim \|f\|_{L_x^{p}(\mathbb{R}^2)}$ and $\|\chi_{B(0,\delta \epsilon)}\|_{L_x^{p}(\mathbb{R}^2)} \approx \epsilon^{\frac{2}{p}}$, we obtain $\epsilon^{\frac{1}{q} - \frac{2}{p} + 1} \lesssim 1$ for all $\epsilon > 0$. For sufficiently small $\epsilon$, this requires $\frac{1}{q} \geq \frac{2}{p} - 1$.

On the other hand, for $\epsilon > 0$, define
\begin{align*}
Q := \left\{x \in \mathbb{R}^2 : |x \cdot e_1| \leq 6 \epsilon \text{ and } |x \cdot e_2| \leq [3 + 2|\gamma''(1)|] \epsilon^2\right\},
\end{align*}
where $e_1 := \frac{(-1, -\gamma'(1))}{\sqrt{1 + \gamma'(1)^2}}$ and $e_2 := \frac{(-\gamma'(1), 1)}{\sqrt{1 + \gamma'(1)^2}}$. Additionally, define
\begin{align*}
\Omega_{u_i} := \left\{x \in \mathbb{R}^2 : \left|\left(x - (u_i, u_i)\right) \cdot e_1\right| \leq \epsilon \text{ and } \left|\left(x - (u_i, u_i)\right) \cdot e_2\right| \leq \epsilon^2\right\},
\end{align*}
where $\{u_i\} \subset [1,2]$ satisfies $|u_{i+1} - u_i| = \frac{4\epsilon^2}{|(1,1)\cdot e_2|}$ for all $i$. Moreover, for sufficiently small $\epsilon > 0$, we obtained in \cite{LLLY} that
\begin{align*}
T\chi_Q(x,u) \geq \int_{1 - \frac{\epsilon}{\sqrt{1 + \gamma'(1)^2}}}^1 \chi_Q(x_1 - ut, x_2 - u\gamma(t)) \, \mathrm{d}t = \frac{\epsilon}{\sqrt{1 + \gamma'(1)^2}}
\end{align*}
for all $x \in \Omega_{u_i}$ and $u \in (u_i - \epsilon^2, u_i + \epsilon^2)$. Since $\Omega_{u_i} \cap \Omega_{u_{i'}} = \emptyset$ for $i \neq i'$ and the number of elements in $\{u_i\} \subset [1,2]$ is equivalent to $ \epsilon^{-2}$, we conclude
\begin{align*}
\left\|T\chi_Q\right\|_{L^r_u L_x^q(\mathbb{R}^2 \times [1,2])} \geq \left\{\sum_i \int_{u_i - \epsilon^2}^{u_i + \epsilon^2} \left[\int_{\Omega_{u_i}} \left|T\chi_Q\right|^q \, \mathrm{d}x\right]^{\frac{r}{q}} \, \mathrm{d}u\right\}^{\frac{1}{r}} \gtrsim \epsilon^{1 + \frac{3}{q}}.
\end{align*}
Combining this with $\|\chi_Q\|_{L^{p}(\mathbb{R}^2)} \approx \epsilon^{\frac{3}{p}}$ and $\|Tf\|_{L^r_u L_x^q(\mathbb{R}^2 \times [1,2])} \lesssim \|f\|_{L_x^{p}(\mathbb{R}^2)}$, we obtain $\epsilon^{1 + \frac{3}{q} - \frac{3}{p}} \lesssim 1$ for all sufficiently small $\epsilon > 0$. Letting $\epsilon \to 0$ implies $\frac{1}{q} \geq \frac{1}{p} - \frac{1}{3}$ for all $r \in [1, \infty]$.

Thus, we obtain $\frac{1}{q} \geq \max\{ \frac{2}{p} - 1, \frac{1}{p} - \frac{1}{3} \}$, and (III) of Theorem \ref{thm3} is proved.

\subsubsection{Proof of (IV) of Theorem \ref{thm3}}

From \eqref{eq:5.3}, setting $\bar{s}_{2,-\frac{1}{2}}(p,q) < 0$ yields a decay estimate $F_{j,k} : L_x^p(\mathbb{R}^2) \to L^{\infty}_u L_x^q(\mathbb{R}^2 \times \mathbb{R})$, where $(\frac{1}{p}, \frac{1}{q})$ lies in the open region with vertices $(0, 0)$, $(\frac{2}{3}, \frac{1}{3})$, and $(1, 1)$. Then, $T : L_x^p(\mathbb{R}^2) \to L^{\infty}_u L_x^q(\mathbb{R}^2 \times \mathbb{R})$ holds on the intersection of this region with $1 + (1 + \omega)\left(\frac{1}{q} - \frac{1}{p}\right) > 0$. Thus, $T : L_x^p(\mathbb{R}^2) \to L_x^q(\mathbb{R}^2)$ holds on this region uniformly in $u \in [1,2]$. Combining this with \cite[Theorem 2.1]{LLLY}, we conclude that $(0, 0)$, $(\frac{2}{3}, \frac{1}{3})$, and $(1, 1)$ are the endpoints of the estimate $F_{j,k} : L_x^p(\mathbb{R}^2) \to L^{\infty}_u L_x^q(\mathbb{R}^2 \times \mathbb{R})$. Therefore, $\bar{s}_{2,-\frac{1}{2}}(p,q)$ is the infimum ensuring \eqref{eq:5.3} holds.

On the other hand, as in Subsection 5.1, for $1 \leq p \leq q \leq \infty$, the $L_x^p(\mathbb{R}^2) \to L^{\infty}_u L_x^q(\mathbb{R}^2 \times \mathbb{R})$ estimates for $E_{j,k} - F_{j,k}$ and $\widetilde{T}^{1,a}_{j,k}$ are bounded by $2^{-Nk}$ for large $N$. Recall that $Q_k f = (\varrho_k(\xi)\hat{f}(\xi))^{\vee}$, and a similar estimate to \eqref{eq:5.3} holds. Namely, if
\begin{align*}
\left\|\widetilde{T}^1_{j,k} Q_k f\right\|_{L^{\infty}_u L_x^q(\mathbb{R}^2 \times \mathbb{R})} \lesssim 2^{s k} \|f\|_{L_x^{p}(\mathbb{R}^2)},
\end{align*}
then $s \geq \bar{s}_{2,-\frac{1}{2}}(p,q)$. Now, from Lemma \ref{lemma5.a} and \eqref{eq:5.a}, we have
\begin{align*}
\left\|\widetilde{T}^1_{j,k} Q_k f\right\|_{L^{\infty}_u L_x^q(\mathbb{R}^2 \times \mathbb{R})} \lesssim 2^{\frac{1}{r}k} \left\|\widetilde{T}^1_{j,k} Q_k f\right\|_{L^r_u L_x^q(\mathbb{R}^2 \times \mathbb{R})} \lesssim 2^{\frac{1}{r}k} \|f\|_{L_x^{p}(\mathbb{R}^2)},
\end{align*}
which implies $\frac{1}{r} \geq \bar{s}_{2,-\frac{1}{2}}(p,q)$, i.e., $$\frac{1}{q} \geq \max\left\{\frac{1}{2p} - \frac{1}{2r}, \frac{2}{p} - \frac{r+1}{r}\right\}.$$

From \eqref{eq:5.2}, setting $s_{-\frac{1}{2},0}(p,q,q) < 0$ yields the decay estimate $F_{j,k} : L_x^p(\mathbb{R}^2) \to L^{\infty}_u L_x^q(\mathbb{R}^2 \times \mathbb{R})$ when $(\frac{1}{p}, \frac{1}{q})$ lies in the open region with vertices $(0, 0)$, $(\frac{1}{2}, \frac{1}{6})$, $(\frac{2}{3}, \frac{1}{3})$, and $(1, 1)$. This result is nearly sharp by \cite{LLLY}. Similarly, if
\begin{align*}
\left\|\widetilde{T}^1_{j,k} Q_k f\right\|_{L^{q}_u L_x^q(\mathbb{R}^2 \times \mathbb{R})} \lesssim 2^{s k} \|f\|_{L_x^{p}(\mathbb{R}^2)},
\end{align*}
then $s \geq s_{-\frac{1}{2},0}(p,q,q)$. On the other hand, by \eqref{eq:5.a}, applying the Sobolev embedding theorem with respect to $u$ for $q > r$ and $\textrm{H}\ddot{\textrm{o}}\textrm{lder}$'s inequality for $q \leq r$ yields
\begin{align*}
\left\|\widetilde{T}^1_{j,k} Q_k f\right\|_{L^q_u L_x^q(\mathbb{R}^2 \times \mathbb{R})} \lesssim 2^{\max\left\{\frac{1}{r} - \frac{1}{q}, 0\right\}k} \|f\|_{L_x^{p}(\mathbb{R}^2)}.
\end{align*}
It further yields $\max\{\frac{1}{r} - \frac{1}{q}, 0\} \geq s_{-\frac{1}{2},0}(p,q,q)$.
Note that \( s_{-\frac{1}{2},0}(p,q,q) = \tilde{s}_{A,0}(p,q,q) \), we obtain $$\max\left\{\frac{1}{r} - \frac{1}{q}, 0\right\} \geq \tilde{s}_{A,0}(p,q,q).$$

Putting these results together, (IV) of Theorem \ref{thm3} is proved.

\bigskip

\noindent{\bf Acknowledgments.}
The authors would like to thank N. Liu for helpful suggestions and discussions. J. Li is supported by National Natural Science Foundation of China (No.~12471090). Z. Lou is supported by National Natural Science Foundation of China (No.~12471093), Guangdong Basic and Applied Basic Research Foundation (No.~2024A1515010468), and LKSF STU-GTIIT Joint-research Grant (No.~2024LKSF06). H. Yu is supported by National Natural Science Foundation of China (No.~12201378) and Guangdong Basic and Applied Basic Research Foundation (No.~2023A1515010635).

\bigskip

\noindent  Junfeng Li

\smallskip

\noindent  School of Mathematical Sciences, Dalian University of Technology, Dalian, 116024,  People's Republic of China

\smallskip

\noindent {\it E-mail}: \texttt{junfengli@dlut.edu.cn}

\bigskip

\noindent Zengjian Lou and Haixia Yu (Corresponding author)

\smallskip

\noindent  Department of Mathematics, Shantou University, Shantou, 515821, People's Republic of China

\smallskip

\noindent {\it E-mails}: \texttt{zjlou@stu.edu.cn} (Z. Lou)

\noindent\phantom{{\it E-mails:}} \texttt{hxyu@stu.edu.cn} (H. Yu)

\end{document}